\documentclass{amsart}

\usepackage{graphicx}
\usepackage{latexsym,color}
\usepackage{graphicx}
\usepackage{amssymb}
 \usepackage{amsfonts}
 \usepackage{amsmath}
\usepackage{amsthm}

\newtheorem{thm}{Theorem}[section]
 \newtheorem{dfn}[thm]{Definition}
 \newtheorem{lem}[thm]{Lemma}
 \newtheorem{rem}[thm]{Remark}

\newtheorem{prop}[thm]{Proposition}
\newtheorem{asm}[thm]{Assumption}
\newtheorem{prp}[thm]{Proposition}

\numberwithin{equation}{section}
\newcommand{\cA}{\mathcal{A}}

\newcommand{\cC}{\mathcal{C}}
 \newcommand{\cD}{\mathcal{D}}

\newcommand{\cF}{\mathcal{F}}

\newcommand{\cL}{\mathcal{L}}
\newcommand{\cM}{\cD}

\newcommand{\cZ}{\mathcal{Z}}

\def \D{\mathbb{D}}
\def \E{\mathbb{E}}

\def \L{\mathbb{L}}
 \def \M{\mathbb{M}}
 
 \def\P{\mathbb{P}}
\def \Q{\mathbb{Q}}
\def \R{\mathbb{R}}

\def \Sb{\mathbb {S}}

\def\reff#1{{\rm(\ref{#1})}}

\begin{document}
\title{Martingale Optimal Transport\\ and \\ Robust Hedging  in Continuous Time}

 \thanks{Research partly supported by the
European Research Council under the grant 228053-FiRM,
 by the ETH Foundation and by the Swiss Finance Institute.
} \author{Yan Dolinsky \address{
 Department of Statistics, Hebrew University of Jerusalem, Israel. \hspace{10pt}
 {e.mail: yan.dolinsky@mail.huji.ac.il}}}
\author{H.Mete  Soner \address{
 Department of Mathematics, ETH Zurich \&
 Swiss Finance Institute. \hspace{10pt}
 {e.mail: hmsoner@ethz.ch}}\\
  ${}$\\
 Hebrew University of Jerusalem
 and  ETH Zurich}

\date{\today} \begin{abstract}
The duality between the robust
(or equivalently, model independent)
hedging of path dependent European options and a
martingale optimal transport problem is proved.
The financial market is modeled
through a risky asset whose price is only assumed to be a continuous
function of time. The hedging problem is to construct a minimal super-hedging 
portfolio that consists of dynamically trading the underlying risky asset and
a static position of  vanilla options which can be exercised at the given, fixed 
maturity. The dual is a Monge-Kantorovich type martingale transport
problem of maximizing the expected value of the option over all martingale 
measures that have a given marginal at maturity. In addition to duality, a
family of simple, piecewise constant super-replication portfolios that asymptotically 
achieve  the minimal super-replication cost is constructed.
\end{abstract}

\subjclass[2010]{91G10, 60G44}
 \keywords{European Options, Robust Hedging, Min--Max Theorems, Prokhorov Metric, Optimal transport}%

\maketitle \markboth{Y.Dolinsky and H.M.Soner}{Martingale Optimal Transport}
\renewcommand{\theequation}{\arabic{section}.\arabic{equation}}
\pagenumbering{arabic}

\section{Introduction}\label{sec:1}\setcounter{equation}{0} The original transport 
problem proposed by Monge \cite{M}
 is to optimally move   a pile of soil  to an excavation.
 Mathematically, given two measures
$\nu$ and $\mu$ of equal mass, we look for an optimal bijection of $\R^d$ which 
moves $\nu$ to $\mu$, i.e., look for a map $S$ so that $$ \int_{\R^d}
\varphi(S(x)) d\nu(x) = \int_{\R^d} \varphi(x) d\mu(x), $$ for all continuous functions 
$\varphi$. Then,  with a given cost function $c$, the objective is
to minimize $$ \int_{\R^d}\ c(x,S(x))\ d\nu(x) $$ over all bijections $S$.

In  his classical papers \cite{K,K1}, Kantorovich relaxed this problem by
 considering a probability measure on $\R^d \times \R^d$,
 whose marginals agree with $\nu$ and $\mu$,
 instead of a bijection.
 This generalization linearizes the problem.
Hence it allows for an easy existence result
 and enables one to identify its
 convex  dual.
 Indeed, the dual elements are real-valued
 continuous maps $(g, h)$ of $\R^d$
 satisfying the constraint
\begin{equation} \label{e.cons}
 g(x) + h(y) \le c(x,y).
\end{equation}
 The dual objective is
 to maximize
 $$
 \int_{\R^d} g(x) \ d\nu(x) +
  \int_{\R^d} h(y) \ d\mu(y)
$$ overall $(g, h)$ satisfying the constraint \reff{e.cons}. In the last decades an impressive theory has been developed and we refer the reader to
\cite{AP,V,V9} and to the references therein.

In robust hedging problems, we are  also given two measures. Namely, the initial 
and the final distributions of a stock process. We then construct an
optimal connection. In general, however, the cost functional depends on the whole 
path of this connection and not simply on the final value. Hence, one
needs to consider processes instead of simply the maps $S$. The probability 
distribution of this process has prescribed
 marginals at final
and initial times. Thus, it is in direct analogy with the Kantorovich measure. But, 
financial considerations restrict the process to be a martingale (see
Definition \ref{d.measure}). Interestingly, the dual also has a financial interpretation 
as a robust hedging (super-replication) problem.  Indeed, the
replication constraint is similar to \reff{e.cons}. The formal connection between the 
original Monge-Kantorovich problem and the financial problem is
further discussed in Remark \ref{r.analogy} and also in the papers \cite{BHLP} and 
\cite{GLT}.

We continue by describing the robust hedging problem. Consider a financial market 
consisting of one risky asset with a continuous price process. As in the
classical paper of Hobson  \cite{H}, all call options are liquid assets and can be 
traded for a ``reasonable" price that is known initially.  Hence, the
portfolio of an investor consists of static positions in the call options in addition to the 
usual dynamically updated risky asset. This leads us to a
similar structure to that in \cite{H} and in other papers
\cite{BHR,CL,CO,CO1,DH, DOR, GLT,H2,H3,H4,H5,H6,LT} which consider
model-independent pricing.  
This approach is very closely
related to path-wise proofs 
of well-known probabilistic inequalities
\cite{ABPST,CHO}.
Apart from the continuity of the price process no other 
model assumptions are placed on the dynamics of the price process.

In this market, we prove the Kantorovich duality, Theorem \ref{thm2.1}, and an approximation result, Theorem \ref{thm2.2}, for a general class of
path-dependent options. The classical duality theorem, for a market with a risky asset whose price process is a semi--martingale,
 states that the minimal super-replication
cost of a contingent claim is equal to the supremum of its expected value over all martingale measures that are equivalent to a given measure. We refer the
reader to Delbaen \& Schachermayer \cite{ds} (Theorem 5.7) for the case of general semi-martingale processes and to El-Karoui \& Quenez \cite{eq} for its
dynamic version in the diffusion case. Theorem \ref{thm2.1} below, also provides a dual representation of the minimal super-replication cost but for model
independent markets. The dual is given as  the supremum of the expectations of the contingent claim over all martingale measures with {\em a given marginal
at the maturity but with no dominating measure}.  Since no probabilistic model is pre-assumed for the price process, the class of all martingale measures
is quite large. Moreover, martingale measures
are typically orthogonal to each other. These  facts render the problem difficult. 

In the literature, there are two earlier results in this direction. In a purely discrete setup, a similar result was recently proved by Beiglb\"ock,
Henry-Labord\`ere and Penkner \cite{BHLP}. In their model, the investor is allowed to buy all call options at finitely many given maturities and the stock
is traded only at these possible maturities. In this paper, however,
 the stock
is traded in continuous time together with a static position in the calls with one maturity. In \cite{BHLP} the dual is recognized as a Monge-Kantorovich
type optimal transport for martingale measures
and  the main tool in \cite{BHLP} is a duality result from
optimal transport (see Theorem 2.14 in
\cite{Kel}).

In continuous time, Galichon, Henry-Labord\`ere and Touzi \cite{GLT} prove a
different duality and then use the dual to convert the problem to an optimal
control problem. There are two main differences between our result and the one
proved in \cite{GLT}. The duality result, Proposition 2.1 in \cite{GLT},
states that the minimal super-replication cost is given as the infimum over Lagrange
multipliers and supremum over martingale measures {\em without} the
final time constraint and the Lagrange multipliers are related to the constraint.  Also
the problem formulation is different. The model in \cite{GLT}
assumes a large class of possible martingale measures for the price process. The
duality is then proved by extending an earlier unconstrained result proved
in \cite{STZ}. As in the unconstrained model of \cite{DM,STZ,STZ1}, the super
replication is defined not path-wise but rather probabilistically through
quasi-sure inequalities. Namely, the super-replication cost is the minimal initial
wealth from which one can super-replicate the option almost surely
with respect to all measures in a given class. In general, these measures are not
dominated by one measure.  As already mentioned this is the main
difficulty and sets the current problem apart from the classical duality discussed
earlier. However, our duality result together with the results of
\cite{GLT} implies that these two approaches -- namely, robust hedging through the
path-wise definition of this paper and the quasi-sure definition of
\cite{DM,STZ,STZ1} yield the same value.  This is proved
 in Section \ref{s.quasi} below.

Our second result provides a class of portfolios which are managed on a finite
number of random times and asymptotically achieve the minimal
super-replication cost. This result may have practical implications allowing us to
numerically investigate
 the corresponding discrete hedges,
 but we relegate this to a future study.

Robust hedging has been an active research area over the past decade. The initial
paper of Hobson \cite{H} studies the case of the lookback option. The
connection to the Skorokhod embedding is also made in this paper and an explicit
solution for the minimal super-replication cost is obtained. 
This approach
is further developed by Brown, Hobson and Rogers \cite{BHR}, 
Cox and Obloj \cite{CO}, \cite{CW} and 
in several other papers, \cite{H1,H2,H3,H4,H5,H6}. We refer the reader to the excellent survey of Hobson \cite{H1} for robust
hedging and to Obloj \cite{O} for the
Skorokhod embedding problem.
In particular, the recent paper 
by Cox and Wang \cite{CW} provides
a discussion of various constructions of Root's solution of the Skorokhod
embedding.

A similar modeling
approach is applied to volatility options by Carr and Lee \cite{CL}. 
In a recent paper, Davis, Obloj and Raval \cite{DOR} considers
 the variance swaps in a market with finitely many put options. In particular,  in
\cite{DOR} the  class of admissible portfolios is enlarged and numerical evidence is obtained by analyzing the S\&P500 index options data.

As already mentioned above, the dual approach is used by 
Galichon, Henry-Labord\`ere and Touzi  \cite{GLT} 
and Henry-Labord\`ere and Touzi \cite{LT} as
well. In these papers, the duality provides a connection to stochastic optimal control 
which can then be used to compute the solution in a more systematic
manner.

The proof of the main results is done in four steps. The first step is to reduce the
problem to bounded claims. The second step is to represent
the original robust hedging problem as a limit of robust hedging problems which live
on a sequence of countable spaces. For these type of problems, robust
hedging is the same as classical hedging, under the right choice of a probability
measure. Thus we can apply the classical duality results for
super--hedging of European options on a given probability space. The third step is to
use the discrete structure and  apply a standard min--max theorem
(similar to the one used in \cite{BHLP}). The last step is to analyze the limit of the
obtained prices in the discrete time markets. We
combine methods from arbitrage--free pricing and limit theorems for stochastic
processes.

The paper is organized as follows. The main results are formulated in the next
section. In Section \ref{s.quasi}, the connection between the quasi sure
approach and ours is proved. The two sections that follow are devoted to the proof of one inequality which implies the main results. 
The final section discusses a possible extension.
\vspace{10pt}

\noindent {\bf{Acknowledgements.}} The authors would like to thank
Mathias Beiglb\"ock,
 David Belius, Alex Cox, Zhaoxu Hou, Jan Obloj,
 Walter Schachermayer and
 Nizar Touzi
for insightful discussions and comments.
In particular, Section \ref{s.quasi} resulted from discussions with
Obloj and Touzi and we are grateful to Hou, Obloj and Shachermayer for several
corrections.

\section{Preliminaries and main results} \label{sec2} \setcounter{equation}{0} The
financial market consists of
 a savings account
 which is normalized to unity
 $B_t\equiv 1$
by discounting and of a risky asset $S_t$, $t\in [0,T]$, where $T<\infty$ is the maturity
date. Let  $s:=S_0>0$ be the initial stock price and without loss
of generality, we set $s=1$. Denote by $\mathcal{C}^{+}[0,T]$ the set of all strictly
positive functions $f:[0,T]\rightarrow\mathbb{R}_{+}$ which satisfy
$f_0=1$. We assume that $S_t$ is a continuous process. Then, any element of
$\mathcal{C}^{+}[0,T]$ can be a possible path for the stock price process $S$.
Let us emphasize that this the only assumption that we make on our financial market.

Denote by $\cD[0,T]$ the space of all measurable functions $\upsilon:[0,T]\rightarrow
\mathbb{R}$ with the norm $||\upsilon||=\sup_{0\leq t \leq
T}|\upsilon_t|$. Let $G:\cD[0,T]\rightarrow \mathbb{R}$ be a given deterministic map.
We then consider a path dependent European option with the payoff
\begin{equation}\label{2.1} X=G(S), \end{equation} where $S$ is viewed as an
element in $\cD[0,T]$. \vspace{10pt}

\subsection{An assumption on the claim}

Since our proof is through an approximation argument, we need
the  regularity of the pay-off functional $G$.  Indeed, we first approximate
the stock price process by piece-wise constant  functions
taking values in a finite set.  We also discretize the jump
times to obtain a countable set of possible price processes.
This process necessitates a continuity assumption
with respect to a Skorokhod type topology.
Further discussion of this assumption is given in Remark
\ref{rem2.1}. In particular,  Asian and lookback type options satisfy the below
condition. A possible generalization of our result
to more general class of pay-offs is discussed in
the final section
\ref{s.extension}.

Let $\cD_N[0,T]$ be the subset of
$\cD[0,T]$ that are piece-wise constant functions with $N$ possible jumps i.e., $v \in
\cD_N[0,T]$ if and only if there exists a partition $ t_0=0 < t_1 <t_2<\ldots<t_N \le  T$
such that
$$
v_t= \sum_{i=1}^{N}\ v_i \chi_{[t_{i-1},t_i)}(t) +
v_{N+1}\chi_{[t_{N},T]}(t) , \ \ {\mbox{where}}\ \  v_i:= v_{t_{i-1}},
$$
and we set
$\chi_A$
be the characteristic function of the set $A$. We make
the following standing assumption on $G$.

\begin{asm}
\label{asm2.1}
There exists a constant $L>0$ so that
$$
|G(\omega)-
G(\tilde\omega)|\leq L \|\omega-\tilde\omega\|,  \ \ \omega,\tilde\omega\in
\cD[0,T],
$$
where as before, $||\cdot||$ is the $\sup$ norm.

Moreover, let $\upsilon, \tilde\upsilon \in\cM_N[0,T]$ be such that $\upsilon_{i}=\tilde
\upsilon_{i}$ for all $i=1,...,N$. Then, $$
|G(\upsilon)-G(\tilde\upsilon)| \leq L \|\upsilon\|\sum_{k=1}^{N}|\Delta t_k-\Delta \tilde
 t_k|, $$ where as usual $\Delta t_k := t_k- t_{k-1}$ and $\Delta
 \tilde t_k := \tilde t_k- \tilde t_{k-1}$. \end{asm}

\begin{rem}\label{rem2.1} {\rm{ In our setup, the process $S$ represents the
discounted stock price and $G(S)$ represents the discounted reward. Let $r>0$
be the constant interest rate. Then, the payoff $$ G(S):=e^{-rT}H\left(e^{rT}S_T,
\min_{0\leq t\leq T} e^{rt}S_t, \max_{0\leq t\leq T} e^{rt}S_t,\int_{0}^T
e^{rt} S_t dt\right), $$ with  a Lipschitz continuous function
$H:\mathbb{R}^4\rightarrow\mathbb{R}$ satisfies the above assumption.

The above condition on $G$ is, in fact, a Lipschitz assumption with respect to a
 metric very similar to the Skorokhod one. However, it  is {\em{weaker}}
than to assume Lipschitz continuity with respect to the Skorokhod metric. Recall that
this classical metric is given by
$$
d(f,g):=\inf_{\lambda}\sup_{0\leq t\leq T} \max\left(|f(t)-g(\lambda(t))|,|\lambda(t)-t|
\right),
$$
where the infimum is taken over all time changes. A
time change is a strictly increasing continuous function which satisfy $\lambda(0)=0$
and $\lambda(T)=T$. We refer the reader to  Chapter 3 in \cite{B} for
more information. In particular, while  $\int_{0}^T S_t dt$ is continuous with respect to
the Skorokhod metric in $\cC[0,T]$, it is not Lipschitz
continuous in $\cC[0,T]$ and it is not even continuous in $\cD[0,T]$. Although we
assume $S$ to be continuous, since in our analysis we need to consider
approximations in $\cD[0,T]$, the above assumption is needed in order to include
Asian options.

Moreover, from our proof of the main results it can be shown that Theorems
\ref{thm2.1} and \ref{thm2.2} can be extended to payoffs of the form
$$
e^{-rT}H\left(e^{r t_1}S_{t_1},..., e^{r t_k} S_{t_k}, \min_{0\leq t\leq T}
e^{rt} S_t, \max_{0\leq t\leq T}e^{rt}S_t, \int_{0}^T e^{rt} S_t dt\right)
$$
where $H$ is Lipschitz and $0<t_1<...<t_k\leq T$. \qed}} \end{rem}

\vspace{10pt}

\subsection{European Calls} \label{ss.call}

We assume that, 
at time zero, the
investor is able to buy any call option with strike $K\ge 0$, for the
price
\begin{equation}
\label{e.c}
C(K):= \int \left(x-K\right )^{+}d\mu(x),
\end{equation}
where $\mu$ is a given probability measure on $\mathbb{R}_{+}$. 
The  measure $\mu$ is
assumed to be derived from observed call prices that are liquidly
traded in the market. One may also think of $\mu$ as describing the probabilistic
belief (in the market) about the stock price distribution at time $T$.
Then, an approximation argument implies that the price of a derivative security with
the payoff $g(S_T)$ with a bounded, measurable $g$ must be given by
$\int g d\mu$. We then assume that this formula also holds for all
$g \in \L^1(\R_+,\mu)$.

In particular, $C(0)= \int x d\mu(x)$.  On the other hand
the pay-off $C(0)$ is one stock. Hence, 
the value of $C(0)$ must be equal to the initial stock price
$S_0$ which is normalized to one. Therefore, although the
probability measure $\mu$ is quite general,
in view of our assumption \reff{e.c} and arbitrage considerations,
it should satisfy 
\begin{equation} 
\label{2.2}
C(0)=\int xd\mu(x)=S_0=1. 
\end{equation} 
For technical reasons, we also assume
that there exists $p>1$ such that
\begin{equation}
\label{2.3} \int x^p
d\mu(x)<\infty.
\end{equation}
\vspace{10pt}

\subsection{Admissible portfolios}
\label{ss.admissible}

We continue by describing the continuous time trading in the underlying asset $S$.
Since we do not assume any semi-martingale structure of the risky asset,
this question is nontrivial. We adopt the path-wise approach and require
that the trading strategy (in the risky asset) is of finite variation.
Then, for any function $h:[0,T]\rightarrow \mathbb{R}$ of finite variation and
continuous function $S \in \cC[0,T]$, we use integration by parts to define
$$ \int_{0}^t h_u dS_u:= h_t S_t-h_0S_0-\int_{0}^t S_udh_u, $$ where the last term
in the above right hand side is the standard Stieltjes integral.

We are now ready to give the definition of 
semi-static portfolios and super-hedging.
Recall  the exponent $p$ in (\ref{2.3}).

\begin{dfn}
\label{dfn2.1}

{\rm{1.  We say that a map $$ \phi : A \subset \cD[0,T]\rightarrow\cD[0,T] $$ is}}
progressively measurable, {\rm{if for any $v, \tilde v \in A$,
\begin{equation} \label{e.adapted} v_u=\tilde{v}_u, \ \ \forall u \in[0,t] \ \ \Rightarrow \ \
\phi(v)_t=\phi(\tilde v)_t. \end{equation}

\noindent 2. A}} semi-static portfolio {\rm{is a pair $\pi:=(g,\gamma)$, where $g\in
\L^1(\R_+,\mu)$ and
$$
\gamma:\mathcal{C}^{+}[0,T]\rightarrow\cD[0,T]
$$
is a progressively measurable map of bounded variation.

\noindent 3. The corresponding }}discounted portfolio value {\rm{is given by,
$$
Z^{\pi}_t(S)=g(S_T)\chi_{\{t=T\}}+\int_{0}^t \gamma_u( S) d{S}_u, \ \ t\in
[0,T],
$$
where $\chi_{A}$ is the indicator of the set $A$. A semi-static
portfolio is}} admissible, {\rm{if there exists $M>0$ such that
\begin{equation} \label{2.6+} Z^\pi_t(S)\geq -M \left(1+\sup_{0\leq u\leq t}S^p_u
\right), \ \ \ \forall t\in [0,T], \ \ S\in\mathcal{C}^{+}[0,T].
\end{equation}

\noindent 4. An admissible semi-static portfolio is called}} super-replicating, {\rm{if
$$
Z^\pi_T(S)\geq  G(S), \ \ \forall{S}\in \mathcal{C}^{+}[0,T].
$$
Namely, we require that for any possible value of the stock process, the portfolio
value at maturity will be no less that the reward of the European
claim.

\noindent 5. The (minimal)}} super-hedging cost {\rm{of $G$ is  defined by,}}
\begin{equation*}
V(G):=\inf\left\{\int g d\mu: \ \exists \gamma \
\mbox{such}  \ \mbox{that} \ \pi:=(g,\gamma) \ \mbox{is} \ \mbox{super-replicating} \ \right\}.
\end{equation*}
\qed
\end{dfn}

Notice that the set of admissible portfolios depends on the exponent $p$ which appears in the assumption \reff{2.3}. We suppress this possible dependence
to simplify the exposition. \vspace{10pt}

\subsection{Martingale optimal transport}

Since the dual formula refers to a probabilistic structure, we need to introduce that structure as well. Set $\Omega:=\mathcal{C}^+[0,T]$ and let $\mathbb
S=(\mathbb S_t)_{0\leq t\leq T}$ be the canonical process given by $\mathbb S_t(\omega):=\omega_t$, for all $\omega\in\Omega$. Let $\cF_t:=\sigma(\mathbb
S_s,\, 0\leq s\leq t)$ be the canonical filtration (which is not right continuous).

The following class of probability measures are central to our results. Recall that we have normalized the stock prices to have initial value one.
Therefore, the probability measures introduced below need to satisfy this condition as well.

\begin{dfn} \label{d.measure} {\rm{

A probability measure $\Q$ on the space $(\Omega,{\mathcal F})$ is a}} martingale measure,
 {\rm{if
the canonical process $({\mathbb S_t})_{t=0}^T$ is a local martingale with respect to $\Q$ and $\mathbb S_0=1$ $\Q$-a.s.

For a probability measure $\mu$ on $\R_+$, $\mathbb{M}_\mu$ is the set of all martingale measures $\Q$ such that the probability distribution of
$\mathbb S_T$ under $\Q$ is equal to $\mu$}}. \qed \end{dfn}

Note that if $\mu$ satisfies \reff{2.2}, then the canonical process $({\mathbb S_t})_{t=0}^T$ is a {\em martingale} (not only a local martingale) under any
measure $\Q \in \M_\mu$. Indeed, a strict local martingale satisfies
$$
1=\mathbb S_0>\mathbb{E}_{\Q}[ \mathbb{S}_T]=\int x d\mu(x),
$$
and it would be  in
contradiction with \reff{2.2}. We use $\mathbb{E}_{\Q}$ 
to denote the expectation with respect to $\Q$.

\begin{rem}
{\rm{ Observe that (\ref{2.2}) yields that the set $\mathbb{M}_\mu$ is not
empty. Indeed, consider a complete probability space
$(\Omega^W,\mathcal{F}^W,P^W)$ together with a standard one--dimensional
Brownian motion $(W_t)_{t=0}^\infty$, and the natural filtration $\mathcal{F}^W_t$
which is the completion of $\sigma{\{W_s|s\leq{t}\}}$. Then, there exists a function $f:
\mathbb{R}\rightarrow \mathbb{R}_{+}$ such that the probability
distribution of $f(W_T)$ is equal to $\mu$. Define the martingale $M_t:=E^W(f(W_T)|
\mathcal{F}^W_t)$, $t\in [0,T]$. In view of  (\ref{2.2}),  $M_0=1$.
Since $M$ is a Brownian martingale,  it is continuous. Moreover, since $\mu$ has
support on the positive real line,  $f \ge 0$ and consequently,  $M \ge
0$. Then, the distribution of $M$ on the space $\Omega$ is an element in $
\mathbb{M}_\mu$. \qed}} \end{rem}

\begin{rem}\label{arbitrage}{\rm{
Clearly the duality is very closely related to fundamental
theorem of asset pricing, which states the
existence of a  measure $\Q\in\M_\mu$.
Since, as shown in the above remark such measures exist under
our set of assumptions, the market considered
in this paper is arbitrage-free.  Then, a
natural question that arises is whether our assumptions
on the option prices and the measure $\mu$
can be replaced by the assumption of
no-arbitrage. We do not address this very
interesting question in this paper.
However, several recent papers \cite{ABS,BN}
study this question in discrete time. \qed}}
\end{rem}

The following is the main result of the paper. 
An outline of its proof is given in subsection \ref{ss.proof}, below.

\begin{thm}\label{thm2.1} 
Assume that the European claim $G$ satisfies the Assumption \ref{asm2.1} and the probability measure $\mu$ satisfies \reff{2.2}
and \reff{2.3}. Then, the minimal super-hedging cost is given by
$$
V(G)=\sup_{\Q\in\mathbb{M}_\mu} \mathbb{E}_{\mathbb {Q}}\left[G(\mathbb
S)\right].
$$
\end{thm}

\begin{rem}\label{example}{\rm{
The above theorem provides a
duality result for the robust semi-static hedging
of a general pay-off.
Many specific
examples have been considered in the literature.
Indeed, the initial paper
of Hobson \cite{H} explicitly provides the hedge
for a lookback option.
Similarly,
using the random time change and Skorokhod
embedding method,
\cite{BHR,CL,CO}
and several other papers
analyze barrier options, lookback options and volatility options.
Also the path-wise
proof of the Doob's maximal inequality
given in \cite{ABPST} constructs an explicit
portfolio which robustly hedges
the power of the running maximum.
We use this  hedge in the proof
of Lemma \ref{l.bound} as well.
\qed}}
\end{rem}

\begin{rem}
\label{r.analogy}
{\rm{ One may consider the maximizer, if exists, of the expression
$$
\sup_{\Q\in\mathbb{M}_\mu}
 \mathbb{E}_{\mathbb{Q}}\left[G(\mathbb S)\right],
 $$
 as the {\em optimal transport} of the initial probability measure $\nu=\delta_{\{1\}}$ to
 the final distribution $\mu$.
However, an additional constraint that the connection is a martingale is imposed.
This in turn places a restriction on the measures, namely \reff{2.2}. The
penalty function $c$ is replaced by a more general functional $G$. In this context,
one may also consider general initial distributions $\nu$ rather than
Dirac measures.  Then, the martingale measures with given marginals  corresponds
to the  Kantorovich generalization of the mass transport problem.

The super-replication problem is also analogous to the {\em Kantorovich dual}.
However, the dual elements reflect the fact that the cost functional depends
on the whole path of the connection.

The reader may also consult \cite{BHLP} for a very clear discussion of the connection
between robust hedging and optimal transport.
\qed}}
\end{rem}

\vspace{10pt}

\subsection{A discrete time approximation}
 \label{ss.app}
 Next we construct a special class of simple strategies which achieve asymptotically
 the super--hedging cost $V$.

For a positive integer $N$ and{ any $S\in\mathcal{C}^{+}[0,T]$, set
$\tau^{(N)}_0( S)=0$. Then, recursively define
\begin{equation}
\label{e.tau}
\tau^{(N)}_k(S)=\inf\left\{t>\tau^{(N)}_{k-1}(S): |S_t-S_{\tau^{(N)}_{k-1}( S)}|
=\frac{1}{N}
\right\}\wedge T,
\end{equation}
where we set
$\tau^{(N)}_k(S)=T$, when the above set is empty. Also, define
\begin{equation}
\label{e.H}
H^{(N)}(S)=\min\{k\in\mathbb{N}:\tau^{(N)}_k(S)=T\}.
\end{equation}
Observe that for any $ S\in\mathcal{C}^{+}[0,T]$, $H^{(N)}( S)<\infty$.

Denote by $\mathcal{A}_N$ the set of all portfolios for which the trading in the stock
occurs only at the moments
$0=\tau^{(N)}_0(S)<\tau^{(N)}_1(S)<...<\tau^{(N)}_{H^{(N)}(S)}(S)=T$. Formally,
$\pi:=(g,\gamma)\in\mathcal{A}_{N}$, if it is progressively measurable in
the sense of \reff{e.adapted} and it is of the form
$$
\gamma_t(S) =
\sum_{k=0}^{H^{(N)}(S)-1} \gamma_k(S) \chi_{(\tau^{(N)}_k( S),
\tau^{(N)}_{k+1}(S)]}(t),
$$
for some $\gamma_k(S)$'s. Note that, $\gamma_k(S)$ can depend on $S$ only
through its values up to time $\tau^{(N)}_k( S)$, so that $\gamma_t$
is progressively measurable. Set

$$
V_{N}(G):=\inf\left\{\int g d\mu:  \ \exists \gamma \
 \mbox{such that} \ \pi:=(g,\gamma)\in\mathcal{A}_{N}\
\mbox{is super-replicating}\right\}.
$$

It is clear that for any integer $k \ge 1$, $V_N(G) \geq V_{kN}(G) \geq V(G)$. The
following result proves the convergence to $V(G)$.  This approximation
result  is the second main result of this paper.  Also, it is the key analytical step  in the
proof of duality.

\begin{thm}
\label{thm2.2}
Under the assumptions of Theorem \ref{thm2.1},
$$
\lim_{N\rightarrow\infty}V_{N}(G)=V(G).
$$
\end{thm}

\subsection{Proofs of Theorems \ref{thm2.1} and \ref{thm2.2}} \label{ss.proof}

Since $V_N\ge V$, Theorem \ref{thm2.1} and Theorem \ref{thm2.2}
would follow from the following two inequalities,
\begin{equation}
\label{2.10}
\lim
\sup_{N\rightarrow\infty}V_{N}(G)\leq \sup_{\Q\in\mathbb{M}_\mu}
\mathbb{E}_{\Q}\left[G(\mathbb S)\right] \end{equation} and
\begin{equation}
\label{2.10++}
V(G)\geq \sup_{\Q\in\mathbb{M}_\mu} \mathbb{E}_{\Q}\left[G(\mathbb S)\right].
\end{equation}
The first
inequality is the difficult one and it will be proved in Sections 4 and 5. The second
inequality is simpler and we provide its proof here.

Let $\Q\in\mathbb{M}_\mu$ and let $\pi=(g,\gamma)$ be super-replicating. Since
$\gamma$ is progressively measurable in the sense of
\reff{e.adapted}, the stochastic integral
$$
\int_{0}^t \gamma_u(\mathbb S)d\mathbb S_u
$$
is defined with respect to $\Q$. Also $\mathbb{Q}$ is a
martingale measure. Hence, the above stochastic integral is a $\Q$ local--
martingale.  Moreover, from (\ref{2.6+}) we have,
$$
\int_{0}^t
\gamma_u(\mathbb S)d\mathbb S_u \ge -M(1+\sup_{0\leq u\leq t}|\mathbb{S}_t|^p), \ \ t\in [0,T].
$$
Also in view of (\ref{2.3}) and the Doob-Kolmogorov
inequality for the martingale $\mathbb{S}_t$,
 $$
 \mathbb{E}_{\Q}\sup_{0\leq t\leq T}|\mathbb{S}_t|^p \le
 C_p \mathbb{E}_{\Q}|\mathbb{S}_T|^p = C_p \int |x|^p d\mu < \infty.
$$
Therefore, ${\mathbb E}_{\Q}\int_{0}^T \gamma_u(\mathbb S)d\mathbb S_u\leq 0$.
Since $\pi$ is
super-replicating, we conclude that
$$
\mathbb{E}_{\Q}\left[G(\mathbb S)\right] \leq{\mathbb E}_{\Q}\left(
 \int_{0}^T \gamma_u(\mathbb S)d\mathbb S_u
 +g(\mathbb S_T)\right)
 \leq  {\mathbb E}_{\Q} \left[g(\mathbb S_T)\right]
 =\int g d\mu,
$$
where in the last equality we again used
the fact that the distribution of $\mathbb{S}_T$ under
 $\mathbb{Q}$ is equal to $\mu$. This completes  the proof
of the lower bound. Together with \reff{2.10},
which will be proved later,
it also completes the proofs of the theorems.
\qed

\section{Quasi sure approach and full duality} \label{s.quasi}

An alternate approach  to define robust hedging is to use the notion of quasi sure
 super-hedging
 as was done in \cite{GLT,STZ}.
 Let us briefly recall this notion.
Let $\mathcal{Q}$ be the set of all martingale measures $\Q$ on the canonical space $\mathcal{C}^{+}[0,T]$ under which the canonical process $\mathbb{S}$
satisfies $\mathbb{S}_0=1, \Q$-a.s., has quadratic variation and satisfies $\mathbb E_{\Q} \sup_{0\leq t\leq T} \mathbb{S}_t<\infty$. In this
market, an admissible hedging strategy (or a portfolio) is defined as a pair $\pi=(g,\gamma)$, where $g\in\mathbb{L}_1(\mathbb R_{+},\mu)$ and $\gamma$ is
a progressively measurable process such that the stochastic integral $$ \int_{0}^t \gamma_ud\mathbb{S}_u, \ \ t\in [0,T] $$ exists for any probability
measure $\Q \in \mathcal{Q}$ and satisfies (\ref{2.6+}) $\Q$-a.s. We refer the reader to \cite{STZ} for a complete characterization of this class. In
particular,
 one does not restrict the trading strategies to be
of bounded variation. A portfolio $\pi=(g,\gamma)$ is called an (admissible) {\em quasi-sure super-hedge}, provided that $$ g(\mathbb S_T)+\int_{0}^T
\gamma_u d\mathbb{S}_u \geq G(\mathbb S), \ \ \mathbb{Q} \  \mbox{a.s.}, $$
 for all $\Q \in \mathcal{Q}$.
 Then, the minimal super-hedging cost is given by
$$ V_{qs}(G):=\inf\left\{\int g d\mu:  \ \exists \gamma \
 \mbox{such that} \ \pi:=(g,\gamma)\
\mbox{is a quasi-sure super-hedge} \right\}. $$ Clearly, $$ V(G)\geq V_{qs}(G). $$ From simple arbitrage arguments it follows that $$
V_{qs}(G)\geq\inf_{\lambda\in \mathbb{L}^1(\mathbb R_{+},\mu)} \sup_{\Q\in\mathcal{Q}}\mathbb E_{\Q} \left(G(\mathbb S)-\lambda(\mathbb
S_T)+\int\lambda d\mu \right) $$ where we set $\mathbb{E}_{\Q}\xi\equiv-\infty$, if $\mathbb{E}_{\Q}\xi^{-}=\infty$. Since
$\inf\sup\geq\sup\inf$, the above two inequalities yield, \begin{eqnarray*} V(G)&\geq & V_{qs}(G)\\ &\geq&\inf_{\lambda\in
\mathbb{L}^1(\mathbb{R}_{+},\mu)} \sup_{\Q\in\mathcal{Q}} \mathbb E_{\Q}\left(G(\mathbb S)-\lambda(\mathbb S_T) +\int\lambda d\mu \right)\\ &
\geq &\sup_{\Q\in\mathcal{Q}} \inf_{\lambda\in \mathbb{L}^1(\mathbb{R}_{+},\mu)} \mathbb E_{\Q} \left(G(\mathbb S)-\lambda(\mathbb S_T)
+\int\lambda d\mu \right). \end{eqnarray*}

Now if $\Q \in \mathbb{M}_\mu$, then the two terms
involving $\lambda$ are equal.  So we first restrict the measures to the set $\mathbb{M}_\mu$ and then use
Theorem \ref{thm2.1}.  The result is
\begin{eqnarray*}
V(G)&\geq & V_{qs}(G)\geq
\inf_{\lambda\in \mathbb{L}^1(\mathbb{R}_{+},\mu)}
\sup_{\mathbb Q\in\mathcal{Q}} \mathbb E_{\Q}\left(G(\mathbb S)-\lambda(\mathbb S_T)
+\int\lambda d\mu\right)\\ &\ge& \sup_{\Q\in\mathcal{Q}}
\inf_{\lambda\in \mathbb{L}^1(\mathbb{R}_{+},\mu)}
 \mathbb E_{\Q} \left(G(\mathbb S)-\lambda(\mathbb S_T)
 +\int\lambda d\mu \right)\\ &\ge
&\sup_{\Q\in\mathbb{M}_\mu}
\mathbb{E}_{\Q}\left[G(\mathbb S)\right]= V(G).
\end{eqnarray*}
Hence, all terms in the above are equal. We
summarize this in the following which can be seen as the full duality.
\begin{prp}
\label{p.quasi}
Assume that the European claim $G$ satisfies Assumption
\ref{asm2.1} and the probability measure $\mu$ satisfies \reff{2.2},\reff{2.3}.
Then,
\begin{eqnarray*}
V(G)&= & V_{qs}(G)
=\sup_{\mathbb Q\in\mathbb{M}_\mu} \mathbb{E}_{\Q}\left[G(\mathbb S)\right] \\
&=& \inf_{\lambda\in \mathbb{L}^1(\mathbb{R}_{+},\mu)}
\sup_{\mathbb Q\in\mathcal{Q}}
\mathbb E_{\Q}\left(G(\mathbb S)-\lambda(\mathbb S_T)
+\int\lambda d\mu \right)\\ &=& \sup_{\Q\in\mathcal{Q}}
\inf_{\lambda\in \mathbb{L}^1(\mathbb{R}_{+},\mu)}
\mathbb E_{\Q} \left(G(\mathbb S)-\lambda(\mathbb S_T)
+\int\lambda d\mu \right).
\end{eqnarray*}
\end{prp} \vspace{10pt}

\section{Proof of the main results}
\label{sec3} \setcounter{equation}{0}

The rest of the paper is devoted to the proof of \reff{2.10}. \subsection{Reduction to bounded claims}

The following result will be used in two places in the paper. The first place is Lemma \ref{l.bounded} where we reduce the problem to claims that are
bounded from above. The other place is Lemma \ref{lem3.1}.

Consider a claim with pay-off $$ \alpha_{K}(S):= \|S\| \ \chi_{\{\|S\| \geq K\}}+ \frac{ \|S\|}{K}. $$

Recall that $V_N(\alpha_K)$ is defined in subsection \ref{ss.app}.

\begin{lem} \label{l.bound}
$$
\limsup_{K\rightarrow\infty} \ \limsup_{N \rightarrow\infty} V_N(\alpha_{K}) =0.
$$ \end{lem}
\begin{proof}
In this proof,
we always assume that $N >K > 1$. Let $\tau_k=\tau^{(N)}_k(S)$ and $n=H^{(N)}(S)$ be as in \reff{e.tau}, \reff{e.H}, respectively, and set
$$
\theta:=\theta^{(K)}_N(S)=\min\left\{k:S_{\tau_k}\geq K-1 \right\} \wedge n.
$$
Set $c_p:= p/(p-1)$ where $p$ as in \reff{2.3}.
We define a portfolio $(g^{(N,K)},\gamma^{(N,K)})\in\mathcal{A}_N$ as
follows. For $t\in (\tau_k,\tau_{k+1}]$ and $k=0,1,...,n-1$, let
$$
\gamma^{(N,K)}_t(S)=\gamma^{(N,K)}_{{\tau_k}}(S) =-\frac{ p^2}{K(p-1)}\left(\max_{0\leq
i\leq k}S^{p-1}_{\tau_i}\right)
 -\frac{p^2}{(p-1)}\chi_{\{k\geq \theta\}}\
\left(\max_{\theta \leq i\leq k}S^{p-1}_{\tau_i}\right) ,
$$
$$
g^{(N,K)}(x)=\frac{1}{K}(1+\left((c_p x)^p-c_p\right)^{+})+ \left((c_p x)^p-(c_p\left(K-1\right))^p\right)^{+}+\frac{2}{N}.
$$
We use Proposition 2.1 in
\cite{ABPST} and the inequality $x<1+x^p$, $x\in\mathbb{R}_{+}$, to conclude that for any $t\in [0,T]$
$$
g^{(N,K)}(S_t)+\int_{0}^t \gamma^{(N,K)}_u dS_u
\geq \frac{\bar{S}_t}{K} + \bar{S}_t\ \chi_{\{ \bar{S}_t \geq K\}}, $$ where $$ \bar{S}_t:= \max_{0\leq u\leq t}S_u .
$$
Therefore,
$\pi^{(N,K)}:=(g^{N,K)},\gamma^{(N,K)})$ satisfies (\ref{2.6+}) and super-replicates
$\alpha_{K}$. Hence,
$$
V_N(\alpha_{K}) \le \int g^{(N,K)}d\mu.
$$
Also, in view of (\ref{2.3}), $$ \limsup_{K\rightarrow\infty}\ \limsup_{N\rightarrow\infty} \int g^{(N,K)}d\mu=0. $$ These two inequalities complete the
proof of the lemma. \end{proof}

A corollary of the above estimate is the following reduction to claims that are bounded from above.

\begin{lem}
\label{l.bounded} If suffices to prove \reff{2.10} for claims $G$ that are {{non-negative,}} bounded from above and satisfying Assumption
\ref{asm2.1}. \end{lem}

\begin{proof}

{{ We proceed in two steps.  First suppose that \reff{2.10} holds for nonnegative claims that are bounded from above. Then, the conclusions of Theorem
\ref{thm2.1} and Theorem \ref{thm2.2} also hold for such claims.

Now let $G$ be a non-negative claim satisfying Assumption \ref{asm2.1}. For $K>0$, set
$$
G_K:= G \wedge K.
$$ Then, $G_K$ is bounded and \reff{2.10} holds
for $G_K$. Therefore,
$$
\limsup_{N \to \infty} V_N(G_K) \le \sup_{{\mathbb Q}\in \mathbb{M}_\mu} \E_\Q \left[G_K(\Sb)\right] \le \sup_{{\mathbb Q}\in
\mathbb{M}_\mu} \E_\Q \left[G(\Sb)\right].
$$
In view of Assumption \ref{asm2.1},
$$
G(S) \le G(0) + L \| S\|.
$$
Hence, the set $\{G(S) \ge K\}$ is
included in the set $\{ L\|S\|+G(0)\geq K \}$ and
$$
G \le G_K + \left(L\|S\|+G(0)-K\right) \chi_{\{L\|S\|+G(0)\geq K\}}.
$$
By the linearity of the
market, this inequality implies that
$$
V_N(G) \le V_N(G_K) + V_N\left(\left(L\|S\|+G(0)-K\right)\chi_{\{L\|S\|+G(0)\geq K\}}\right).
$$
Moreover, in view
of the previous lemma,
$$
\limsup_{K\rightarrow\infty}\ \limsup_{N\rightarrow\infty} V_N\left(\left(L\|S\|+G(0)-K\right)\chi_{\{L\|S\|+G(0)\geq
K\}}\right)=0.
$$
Using these, we conclude that
$$ \limsup_{N \to \infty} V_N(G)\leq \sup_{{\mathbb Q}\in \mathbb{M}_\mu} \E_\Q \left[G(\Sb)\right].
$$
Hence, \reff{2.10} holds for all functions that are non-negative and satisfy Assumption \ref{asm2.1}. By adding an appropriate constant this results
extends to all claims that are bounded from below and satisfying Assumption \ref{asm2.1}.

Now suppose that $G$ is a general function that satisfies 
Assumption \ref{asm2.1}.  For $c>0$, set
$$
\check{G}_c:= G \vee (-c).
$$ Then, $\check{G}$ is
bounded from below and \reff{2.10} holds, i.e.,
$$
\limsup_{N\to \infty} V_N(G) \le \limsup_{N\to \infty}   V_N(\check{G}_c)
= \sup_{\Q \in \M_\mu}\ \E_\Q
\left[ \check{G}_c\left(\mathbb S\right)\right].
$$
By Assumption \ref{2.1}, $\check{G}_c\left( S\right)\le G\left(S\right) + \check e_c(S)$ where the
error function is
$$
\check e_c(S):=\left(L\|S\| -G(0)-c\right)\chi_{\{L\|S\| -G(0)- c\ge 0\}}(S).
$$
Since $\check e_c \ge 0$ and it satisfies the
Assumption \ref{2.1},
$$
\sup_{\Q \in \M_\mu}\ \E_\Q \left[ \check e_c\left(\mathbb S\right)\right]
= V(\check e_c) = \lim_{N\to \infty} V_N(\hat e_c).
$$
In view of Lemma \ref{l.bound},
$$
\limsup_{c \to \infty} \sup_{\Q \in \M_\mu}\ \E_\Q
\left[ \check e_c\left(\mathbb S\right)\right] = \limsup_{c \to
\infty} \limsup_{N\to \infty} V_N(\check e_c)=0.
$$
We combine the above inequalities to conclude that
\begin{eqnarray*}
\limsup_{N\to \infty} V_N(G) & \le
& \limsup_{c \to \infty} \sup_{\Q \in \M_\mu}\ \E_\Q \left[ \check{G}_c\left(\mathbb S\right)\right] \\ & \le & \sup_{\Q  \in \M_\mu}\ \E_\Q  \left[
G\left(\mathbb S\right)\right]+ \limsup_{c \to \infty} \sup_{\Q \in \M_\mu}\ \E_\Q \left[ \check e_c\left(\mathbb S\right)\right] \\ & = & \sup_{\Q \in
\M_\mu}\ \E_\Q \left[ G\left(\mathbb S\right)\right].
\end{eqnarray*} This exactly \reff{2.10}. }}
\end{proof}

\subsection{A countable class of piecewise constant functions}
In this section, we provide a piece-wise constant approximation of any continuous function
$S$.  Fix a positive integer $N$.  For any $S \in \cC^+[0,T]$, let $\tau^{(N)}_k(S)$ and $H^{(N)}(S)$ be the times defined in \reff{e.tau} and \reff{e.H},
respectively. To simplify the notation, we suppress their dependence on $S$ and $N$ and also set \begin{equation} \label{e.n}
 n=H^{(N)}( S).
 \end{equation}
  We first define the obvious piecewise
 constant approximation $\hat S=\hat S^{(N)}(S)$
 using these
 times.  Indeed, set
\begin{equation}
\label{e.shat}
\hat S_t:= \sum_{k=0}^{n-1} \ S_{\tau_k} \chi_{[\tau_k,
 \tau_{k+1})}(t) +
\left[S_{\tau_{n-1}}+\frac{1}{N} sign( S_T- S_{\tau_{n-1}})\right] \ \chi_{\{T\}}(t).
\end{equation}

The function, that takes $S$ to $\hat S$ is a map of $\cC^+[0,T]$ into the set of  all functions with values in the target set
$$
A^{(N)}=\left\{i/N\ :\
i=0,1,2, \ldots, \right\}.
$$
Indeed, $\hat S$ is behind the definition of the approximating costs $V_N$. However, this set of functions is not countable
as the jump times are not restricted to a countable set. So, we provide yet another approximation by restricting the jump times as well.

Let $\hat\Omega:=\mathbb{D}[0,T]$ be the space of all right continuous functions $f:[0,T]\rightarrow\mathbb{R}_{+}$ with left--hand limits
($c\grave{a}dl\grave{a}g$ functions). For integers $N,k$, let
$$
U_k^{(N)}:= \left\{i/(2^{k}N): i=1,2, \ldots, \right\} \cup \left\{1/(i2^k N) : i=1,2,
\ldots, \right\},
$$
be the sets of possible differences between two consecutive jump times. Next, we define subsets  $\D^{(N)}$ of  $\mathbb{D}[0,T]$.

\begin{dfn} \label{d.DN} {\rm{A function $f \in \mathbb{D}[0,T]$ belongs to $\D^{(N)}$, if  it satisfies the followings,

1. $f(0)\in \{1-1/N,1+1/N\}$,

2.  $f$ is piecewise constant with jumps at times $t_1,...,t_n$, where $$ t_0=0<t_1<t_2<...<t_n<T, $$

3.  for  any $k=1,...,n$, $|f(t_k)-f(t_{k-1})|=1/N$,

4.  for  any $k=1,...,n$, $t_k-t_{k-1}\in  U^{(N)}_k$. }}
\qed
\end{dfn}

We emphasize, in the fourth condition, the dependence of the set $U^{(N)}_k$ on $k$.  So as $k$ gets larger, jump times take values in a finer grid. Also,
for technical reasons we will need that the functions value at $0$ will be equal to $1\pm 1/N$.

We continue by  defining an approximation of
a generic stock price process $S$,
$$
F^{(N)} :  \cC^+[0,T] \to \D^{(N)},
$$
as follow.  Recall $\tau_k=\tau^{(N)}_k(S)$, $n=H^{(N)}(S)$ from above and  also
from  \reff{e.tau}, \reff{e.H}. Set $\hat \tau_0:=0$, $\hat \tau_n=T$
and for $k=1,...,n-1$, define
\begin{eqnarray*}
\hat \tau_k&:=&\sum_{i=1}^k\ \Delta \hat \tau_i,\\
\Delta \hat \tau_i&=& \max\{\Delta t \in U^{(N)}_i: \Delta t< \Delta \tau_i= \tau_i-\tau_{i-1}\},
\qquad i=1,\ldots, n-1.
\end{eqnarray*}
Clearly,  $0=\hat \tau_0<\hat
\tau_1<...<\hat \tau_{n-1}<\hat \tau_n=T$ and $\hat \tau_k < \tau_k$ for all $k=0,\ldots,n-1$.

We are now ready to define $F^{(N)}(S)$.
For $n=1$, set
$$
F^{(N)}(S)\equiv 1 + \frac{1}{N}\ sign(S_T-1),
$$
and for $n>1$, define
 \begin{eqnarray}
 \label{e.F}
 F^{(N)}_t(S)&= &\sum_{k=1}^{n-1} \
 S_{\tau_{k}} \chi_{[\hat \tau_{k-1},\hat \tau_k)}(t) \\
 \nonumber
&& + \left(S_{\tau_{n-1}}+\frac{1}{N} sign( S_T- S_{\tau_{n-1}})\right)
 \ \chi_{[\hat\tau_{{{n-1}}},T]}(t).
 \end{eqnarray}
Observe
that the value of the $k$-th jump of the process $F^{(N)}(S)$ equals to the value of the $(k+1)$-th jump of the discretization $\hat{S}$ of the original
process $S$. Indeed, for $n > 2$,
 \begin{eqnarray}
 \label{e.shift}
 F^{(N)}_{\hat \tau_{m}} - F^{(N)}_{\hat \tau_{m-1}} &=& S_{\tau_{m+1}}  - S_{\tau_{m}}, \qquad
\forall\  m=1,\ldots,n-2,
\end{eqnarray}
and for $n\ge 2$,
$$
F^{(N)}_{\hat \tau_{n-1}} - F^{(N)}_{\hat \tau_{n-2}} = \frac{1}{N}sign\left(S_{T}  -
S_{\tau_{n-1}}\right).
$$
This shift is essential in order to deal with some delicate questions of adaptedness and predictability. We also
recall that the jump times of $\hat S$ are the random times $\tau_k$'s while the jump times of $F^{(N)}(S)$ are $\hat \tau_k$'s and that all these times
depend both on $N$ and $S$. Moreover, by construction, $F^{(N)}(S) \in \D^{(N)}$. But, it may not be progressively measurable as defined in
\reff{e.adapted}. However, we use $F^{(N)}$ only to lift progressively measurable maps defined on $\D^{(N)}$ to the initial space $\Omega = \cC^+[0,T]$ and
this yields progressively measurable maps on $\Omega$. This procedure is defined and the measurability is proved in Lemma \ref{l.mbl}, below.

The following lemma shows that $F^{(N)}$ is close to $S$ in the sense of Assumption \ref{asm2.1}. We also point out that the following result is a
consequence of the particular structure of $\D^{(N)}$ and in particular $U^{(N)}_k$'s.

\begin{lem}
\label{l.1}
Let $F^{(N)}$ be the map defined in \reff{e.F}.  
For any $G$ satisfying the Assumption \ref{asm2.1} with the constant $L$,
$$
\left|G(S) - G(F^{(N)}(S))\right| \le \frac{4 L \|S\|}{N}, \quad \forall\ S\in \cC^+[0,T].
$$
\end{lem}
\begin{proof}
Set
$$
\hat F:=\hat{F}^{(N)}_t(S):=\sum_{k=0}^{n-1} \ S_{\tau_k}
\chi_{[\hat \tau_k,\hat \tau_{k+1})}(t)
+ \left[S_{\tau_{n-1}}
+\frac{1}{N} sign( S_T-
S_{\tau_{n-1}})\right]  \ \chi_{\{T\}}(t).
$$
Observe that $\hat S$ of \reff{e.shat} and $\hat F$ are
like the  functions $\upsilon$ and $\tilde \upsilon$ in that Assumption
\ref{asm2.1}. Hence,
$$
\left|G(\hat S) - G(\hat F)\right| \le  L \|S\| \ \sum_{k=1}^{n} \left|
\Delta \tau _k - \Delta \hat \tau_k\right|.
$$
For $k<n$,
$$
\Delta \hat \tau_k = \max \left\{ \Delta t\in U^{(N)}_k\ :\ \Delta t  < \Delta \tau_k\
\right\}.
$$
The definition of $U^{(N)}_k$ implies that
$$
0 \le \Delta \tau _k - \Delta \hat \tau_k\le \frac{1}{2^k N}, \quad k =1,\ldots, n-1.
$$
Therefore,
\begin{equation}
\label{e.est}
\sum_{k=1}^{n-1} \left| \Delta \tau _k - \Delta \hat \tau_k\right|
\le \sum_{k=1}^{\infty} \frac{1}{2^k N} =
\frac{1}{N}.
\end{equation}

Combining the above inequalities, we arrive at
$$
\left|G(\hat S) - G(\hat F)\right| \le  \frac{L \|S\|}{N}.
$$
Set $F=F^{(N)}(S)$ and directly
estimate that
\begin{eqnarray*}
\left|G(S) - G(F)\right| & \le & \left|G(S) - G(\hat  S)\right| + \left|G(\hat S)
- G(\hat F)\right| + \left|G(\hat F) -
G(F)\right|\\ &\le& L \|S - \hat S \| +   \frac{L \|S\|}{N} + \left|G(\hat F) - G(F)\right|\\
&=&   \frac{3L \|S\|}{N}+\left|G(\hat F) - G(F)\right|.
\end{eqnarray*} Finally, we observe that by construction,
$$
\left\|\hat F-F\right\|\leq\frac{1}{N}, \quad \Rightarrow \quad \left|G(F) - G(\hat F)\right|
\leq\frac{L}{N}.
$$
The above inequalities completes the proof of the lemma.
\end{proof}

\begin{rem} \label{r.1} {\rm{
The proof of the above Lemma provides one of the reasons behind the particular structure of $U_k^{(N)}$.  Indeed,
\reff{e.est} is a key estimate which provides a uniform upper bound for the sum of the differences over $k$.  Since there is no upper bound on $k$, the
approximating set $U_k^{(N)}$ for the $k$-th difference must depend on $k$.  Moreover, it should have a summable structure over $k$.  That explains the
terms $2^k$.

On the other hand, the  reason for the part $\{1/(i2^k N): i=1,2,\dots\}$ in
the definition of
$U_k^{(N)}$ is to make sure that $\Delta \hat \tau_k>0$. For probabilistic
reasons (i.e. adaptability), we want $\hat \tau_k < \tau_k$. This forces us to approximate $ \Delta \tau_k$ by $\Delta \hat \tau_k$ from below. This and
$\Delta \hat \tau_k>0$ would be possible only if  $U_k^{(N)}$
has a subsequence converging to zero.

Hence,  different sets of $U^{(N)}_k$'s are also 
possible provided that they have these two properties.}}
\qed
\end{rem}

\subsection{A countable probabilistic structure} \label{ss.countable}

An essential step in the proof of \reff{2.10} is a duality result for probabilistic problems. We first introduce this structure and then relate it to the
problem $V_N$.

As before, let $\hat\Omega:=\mathbb{D}[0,T]$ be the space of all right continuous functions $f:[0,T]\rightarrow\mathbb{R}_{+}$ with left--hand limits
($c\grave{a}dl\grave{a}g$ functions). Denote by $\hat {\mathbb S}=(\hat {\mathbb S}_t)_{0\leq t\leq T}$ the canonical process on the space $\hat \Omega$.

The set $\D^{(N)}$ defined in Definition \ref{d.DN} is a countable subset of $\hat \Omega$. We choose any probability measure $\hat{\P}^{(N)}$  on
$\hat \Omega$ which satisfies $\hat{\P}^{(N)}\left(\D^{(N)}\right)=1$ and $\hat{\P}^{(N)}(\{f\})>0$ for all $f\in \D^{(N)}$. Let
$\hat{\mathcal{F}}^{(N)}_t$, $t\in [0,T]$ be the filtration
 generated by the process
$\hat{\mathbb S}$ and contains $\hat{\P}^{(N)}$ null sets. Under the measure $\hat{\P}^{(N)}$, the canonical map $\hat{\mathbb S}$ has finitely many
jumps. Let $$ 0= \hat{\tau}_0(\hat{\mathbb S}) <\hat{\tau}_1(\hat{\mathbb S})<\ldots <\hat{\tau}_{\hat{H}(\hat{\mathbb S})}(\hat{\mathbb S})<T, $$ be the
jump times of $\hat{\mathbb S}$.  Note that in Definition \ref{d.DN}, the final jump time is always strictly less than $T$.

A trading strategy on the filtered probability space
$(\hat\Omega,\hat{\mathcal{F}}^{(N)},
(\hat{\mathcal{F}}^{(N)}_t)_{t=0}^T,
\hat{\mathbb
P}^{(N)})$ is a predictable stochastic process $(\hat\gamma_t)_{t=0}^T$. Thus, its a function $\hat\gamma: \D[0,T] \to \cD[0,T]$. Let $a\in \cD[0,T]$ be
such that $a\notin\hat\gamma(\D^{(N)})$. Define a  map $ \phi: \D[0,T] \to \cD[0,T],$ by $\phi(\omega)=\hat\gamma(\omega)$ if $\omega \in
\D^{(N)}$, and equal to $a$ otherwise. Clearly, $\hat{\P}^{(N)}$ almost surely, $\hat\gamma=\phi(\hat{\mathbb S})$. Also, since
$\hat \P^{(N)}$ is non-zero on every point
in $\D^{(N)}$, the definition
of the predictable sigma algebra implies that
$\phi$ is a predictable map. Namely,
for any $v, \tilde v \in \D[0,T]$ and $t\in [0,T]$
$$
v_u=\tilde{v}_u \ \forall u \in[0,t) \ \ \Rightarrow \ \  \phi(v)_t=\phi(\tilde v)_t.
$$ 
Indeed, arguing by contraposition, if there were
$t\in [0,T]$ and $v,\tilde v\in \D^{(N)}$
such that $v_u=\tilde{v}_u$ for all $u \in[0,t)$ and $\phi(v)_t\neq\phi(\tilde
v)_t$. Then, we would conclude that the event 
$\{\hat\gamma_t=\phi(v)_t\}\not\in \hat{\mathcal{F}}^{(N)}_{t-}$.  However, this 
would be 
in contradiction with the predictability of the process $\hat\gamma$. 
(Recall that ${\mathcal{F}}^{(N)}_{t-}$ is the smallest $\sigma$--algebra which contains 
${\mathcal{F}}^{(N)}_{s}$ for any $s<t$).
Hence, any predictable process $\hat \gamma$ has a version $\phi$
that is progressively
measurable in the sense of Definition \ref{dfn2.1}.
In what follows, we always
use this progressively measurable version
of any predictable process. In particular,
the following can be seen as
the probabilistic counterpart of the Definition \ref{dfn2.1}.

\begin{dfn} \label{dfn3.1}  {\rm{1.  A (probabilistic)}} semi-static portfolio
 {\rm{is a pair $(h,\hat \gamma)$
such that $ \hat \gamma :\D[0,T] \to \cD[0,T]$ {{is predictable}} and the stochastic integral $\int_{0}^{\cdot} \hat\gamma_{u} d\hat{\mathbb S}_u$ exists
(with respect to the measure $\mathbb{P}^{(N)}$), and $h:{A^{(N)}}\rightarrow \mathbb{R}$.

\noindent 2. A semi-static portfolio is}} $\hat{\P}^{(N)}$-admissible, {\rm{if $h$ is bounded and there exists $M>0$ such that \begin{equation}
\label{3.0} \int_{0}^t \hat\gamma_{u} d\hat{\mathbb S}_u \ge -M , \ \ \hat{\P}^{(N)}-{a.s.}, \ \ t\in [0,T]. \end{equation}

\noindent 3. An admissible semi-static portfolio is}} $\hat{\P}^{(N)}$-super-replicating, {\rm{if}} \begin{equation} \label{3.1} h(\hat {\mathbb
S}_T)+\int_{0}^T \hat \gamma_{u} d\hat{\mathbb S}_u\geq G(\hat{\mathbb S}),
 \ \ \hat{\P}^{(N)}-{a.s.}
\end{equation} \qed \end{dfn}

\subsection{Approximating $\mu$} \label{ss.mu}

Recall the set $\cA_N$ of portfolios used in the definition of $V_N$ in subsection \ref{ss.app}.

Next we provide a connection between the 
probabilistic super-replication and the discrete robust problem. 
However, the option $h$ in the
Definition \ref{dfn3.1} above is
defined only on $A^{(N)}$ while the static part of the hedges in $\cA_N$ are 
functions defined on $\R_+$.  So for a given $h:A^{(N)} \to \R$, we define the
following operator 
$$ 
g^{(N)}:= \cL^{(N)}(h): \R_+ \to \R 
$$ 
by 
$$ 
g^{(N)}(x) := (1+ \lfloor Nx \rfloor-Nx) h( \lfloor Nx \rfloor/N) 
+ (Nx- \lfloor Nx
\rfloor) h( (1+\lfloor Nx \rfloor)/N), 
$$ where for a real number $r$, $\lfloor r \rfloor$ is the largest integer that is not
larger than $r$.

Next, define a measure $\mu^{(N)}$ on the set $A^{(N)}$ by 
$$ 
\mu^{(N)}(\{0\}):=\int_{0}^{{1/N}} \left(1-N x\right)d\mu(x) 
$$ 
and for any positive integer
$k$,
$$ 
\mu^{(N)}(\{k/N\}):= \int_{{(k-1)/N}}^{{k /N}} \left(N x+1-k\right)d\mu(x)
 + \int_{{k/N}}^{{(k+1)/N}} \left(1+k-N x\right)d\mu(x). 
 $$ 
 This
construction has the following important property. For any  bounded function 
$h: A^{(N)}\to \mathbb{R}$, let $g^{(N)}=\cL^{(N)}(h)$ be as above. Then,
\begin{equation} 
\label{3.17} 
\int h d\mu^{(N)}=\int g^{(N)}d\mu. 
\end{equation} In particular, by taking $h\equiv 1$, we conclude that $\mu^{(N)}$ is a
probability measure. Also, since for continuous $h$, 
$g^{(N)}$ converges pointwise to $h$, one may directly show 
(by Lebesgue's dominated convergence theorem) that $\mu^{(N)}$
converges weakly to $\mu$.

\subsection{Probabilistic super-replication.}

We now introduce the super-replication problem by requiring that the inequalities
 in \reff{3.1} hold $\hat{\P}^{(N)}$-almost surely.
Let $G$ be a European claim as before and $N$ be a positive integer. Then, the 
probabilistic super-replication problem is given by, 
$$ 
\hat{V}_N(G)= \inf
\left\{\int h d\mu^{(N)}: \exists\  \hat{\gamma} \ {\mbox{s.t.}}\
 (h,\hat\gamma)
 \ \mbox{is a}\
 {\hat{\P}^{(N)}}\
 {\mbox{admissible super
 hedge of}}\ G
 \right\}.
$$ 
Recall that in the probabilistic structure, admissibility and related notions are 
defined in Definition \ref{dfn3.1}.

We continue by establishing a connection between the probabilistic super hedging
$\hat{V}_N$ and the discrete robust problem $V_N$.  Suppose that we are
given a probabilistic semi-static portfolio $\hat \pi = (h, \hat \gamma)$ in the sense of 
Definition \ref{dfn3.1}. We  lift this portfolio to a semi-static
portfolio $\pi^{(N)} = (g^{(N)}, \gamma^{(N)}) \in \mathcal{A}_N$. Indeed, let 
$g^{(N)}=\cL^{(N)}(h)$ be as in subsection
\ref{ss.mu} and define
$\gamma^{(N)}:\mathcal{C}^{+}[0,T]\rightarrow\cD[0,T]$ by
$$
\gamma^{(N)}_t(S)= \sum_{k=1}^{{n-1}} \hat \gamma_{\hat\tau_k}
\left(F^{(N)}(S)\right)
\chi_{(\tau_k,\tau_{k+1}]}(t),
$$
where $\tau_k=\tau_k^{(N)}(S)$  are
 as in \reff{e.tau}, $n$ is as in \reff{e.n} and $F^{(N)}(S)$, $\hat \tau_k:=\hat
\tau_k(S)$ are as in \reff{e.F}.  Note that
the random integer $n$ is the number of crossings of magnitude
of no less than $1/N$.  Moreover, by construction it is exactly one more
than the number of jumps of $F^{(N)}$. Also notice that
$$
\gamma^{(N)}_t (S)=0, \qquad \forall\ t \in [0,\tau_1].
$$

\begin{lem}
\label{l.mbl}
For any probabilistic semi-static portfolio
$(h,\hat \gamma)$,  $\gamma^{(N)}$ defined above is progressively 
measurable in the sense of \reff{e.adapted}.
\end{lem}
\begin{proof}
Let $S, \tilde S \in \cC^+[0,T]$ and $t\in [0,T]$ be such that $S_u= \tilde S_u$ for all $u\leq t$ . We
need to show that
$$
\gamma^{(N)}_t(S) = \gamma^{(N)}_t(\tilde S).
$$
Since the above clearly holds for $t=0$ and $t=T$, we may assume that $t\in (0,T)$.
Set
$$
k_t(S):=k_t^{(N)}(S):= \min \{ i\geq 1, :
 \tau^{(N)}_i \geq t\ \}-1,
$$
so that $0{{\le}}k_t(S) < n$ and
$$
t \in (\tau^{(N)}_{k_t(S)}, \tau^{(N)}_{k_t(S)+1}].
$$
It is clear that $k_t(S)=k_t(\tilde S)$. {{If
$k_t(S)=k_t(\tilde S)=0$, then $\gamma^{(N)}_t(S) = \gamma^{(N)}_t(\tilde S)=0$.
So we assume that $k_t(S)>0$}} and use the definition of $\hat \tau_k$ to
conclude that
$$
\theta:=\hat \tau_{k_t(S)}= \hat \tau_{k_t(\tilde S)}(\tilde S).
$$
Since $0<k_t(S)<n$, we have $n>1$ and $ F^{(N)}_t$ is given by \reff{e.F}, i.e.,
$$
 F^{(N)}_t(S)= \sum_{k=1}^{n-1} \
 S_{\tau_{k}} \chi_{[\hat \tau_{k-1},\hat \tau_{k})}(t)
 + \left(S_{\tau_{n-1}}+\frac{1}{N} sign( S_T- S_{\tau_{n-1}})\right)
 \ \chi_{[\hat\tau_{n-1},T]}(t).
$$
Now, for any $u<\theta= \hat \tau_{k_t(S)}= \hat \tau_{k_t(\tilde S)}(\tilde S)$, the 
above definition
implies that
$$
F_u^{(N)}(S)=S_{\tau_k}\ \ F_u^{(N)}(\tilde S)=\tilde S_{\tau_k},
\quad {\mbox{for some}}\quad
k \le k_t(S)=k_t(\tilde S).
$$
Since by definition $\tau_{k_t(S)}(S)<t$, we conclude that
$$
F^{(N)}_u(S) = F^{(N)}_u(\tilde S), \quad \forall \ u \in [0,\theta).
$$
Therefore, by the predictability of $\hat \gamma$ we
have $\gamma^{(N)}_t(S) =\gamma^{(N)}_t(\tilde S)$.
\end{proof}

The following  lemma provides a natural and a crucial
 connection between the probabilistic
super-replication and the discrete robust problem.

Recall the set $\cA_N$ of portfolios used in the definition
of $V_N$ in subsection \ref{ss.app}.

\begin{lem}
\label{lem3.1}
Suppose $G$ is bounded from above and satisfies the Assumption \ref{asm2.1}.
Then,
$$
\limsup_{N\rightarrow\infty}   V_N(G)\leq
\ \limsup _{N\rightarrow\infty} \hat V_N(G).
$$
\end{lem}
\begin{proof}

Set $$ G^{(N)}(S):= G(S)-\frac{5 L \|S\|}{N}. $$ We first show that $$ V_N\left(G^{(N)}\right) \le \hat V_N(G). $$

To prove the above inequality, suppose  that a portfolio $(h, \hat{\gamma})$ is a $
\hat{\P}^{(N)}$-admissible super  hedge of  $G$. Then it suffices
to construct a map $\gamma^{(N)}:\mathcal{C}^{+}[0,T]\rightarrow\cD[0,T]$
and $g^{(N)} : \R_+ \to \R$
such that the semi-static portfolio $\pi^{(N)}:=(g^{(N)},\gamma^{(N)})$ is
admissible, belongs to $\mathcal{A}_N$ and super-replicates $G^{(N)}$ in the sense
of Definition \ref{dfn2.1}.

Let $g^{(N)}= \cL^{(N)}(h)$ be as in subsection \ref{ss.mu} and $\gamma^{(N)}$ be
the probabilistic portfolio considered in Lemma \ref{l.mbl}. We claim that $\pi^{(N)}$
is the desired portfolio. In view of Lemma \ref{l.mbl}, we need  to show that $\pi^{(N)}
$ is in $\mathcal{A}_N$ and super-replicates the $G^{(N)}$ in the
sense of Definition \ref{dfn2.1}.

To simplify the notation, we set $F:= F^{(N)}(S)$. 
\vspace{10pt}

{\em{Admissibility of $\gamma^{(N)}$.}}
By construction trading is only at the random times $\tau_k$'s. Therefore, $\pi^{(N)}
\in \mathcal{A}_N$ provided that it satisfies the lower bound
\reff{2.6+} for every $t \in [0,T]$.
We first
claim that for any $S\in \cC^+[0,T]$ and for every $k \le n-1$,
$$
\int_0^{\tau_k} \gamma^{(N)}_u(S) dS_u = \int_{[0,{\hat
\tau_{k-1}}]} \hat{\gamma}_{u}(F) dF_u.
$$
Since $\gamma^{(N)} \equiv 0$ on $[0,\tau_1]$, the above trivially
holds for $k=1$.  So we assume that $1< k \le n-1$.  In particular,
$n >2$.   Then, we use \reff{e.shift} and the definitions to compute that
\begin{eqnarray*}
 \int_{[0,{\hat \tau_{k-1}}]}
 \hat{\gamma}_{u}(F) dF_u &=&
\sum_{m=1}^{k-1} \hat \gamma_{\hat \tau_m}(F) \left(F_{\hat \tau_m} - F_{\hat 
\tau_{m-1}}\right)= \sum_{m=1}^{k-1} \hat \gamma_{\hat \tau_m}(F)
\left(S_{\tau_{m+1}} - S_{\tau_{m}}\right)\\ 
&=& \sum_{m=1}^{k}  \gamma^{(N)}_{ \tau_{m+1}}(F) 
\left(S_{ \tau_{m+1}} - S_{\tau_{m}}\right)
=\int_{\tau_1}^{\tau_k} \gamma^{(N)}_u(S) dS_u\\ 
&=&\int_{0}^{\tau_k} \gamma^{(N)}_u(S) dS_u.
\end{eqnarray*}
The last identity follows from the fact that
$\gamma^{(N)}$ is zero on the interval $[0,\tau_1]$.

Now, for a given $t \in [0,T{{)}}$ and $ S \in \cC^{+}[0,T]$, let $k{{\le n-1}}$ be the 
largest integer so that $\tau_k \le t$. Construct a function
$\tilde F \in \D^{(N)}$
 by,
$$ 
\tilde F_{[0,\hat \tau_{k})}=F_{[0,\hat \tau_{k})}, \ {\mbox{(i.e.,}} 
\quad \tilde  F_u=F_u, \ \forall \ u \in [0,\hat \tau_{k}),
 {\mbox{)}}
 $$
and
$$
\tilde F_{u}=2F_{\hat\tau_{k-1}}- F_{\hat\tau_{k}}, \ \ u\geq \hat\tau_{k}.
$$
Note that the constructed function $\tilde F$ depends on $S$ and $N$,
since both $F$ and the stopping times $\tau_k$ depend on them.  But we suppress 
these dependences.  Since
$$
\tilde F_{\hat\tau_{k}}- \tilde
F_{\hat\tau_{k-1}} = -\left[F_{\hat\tau_{k}}-F_{\hat\tau_{k-1}}\right] = \pm 1/N,
$$
and since
$$
\left| S_t-S_{\tau_k} \right| \le  1/N,
$$ there exists
$\lambda \in [0,1]$ (depending on $t$)  such that
$$
S_t-S_{\tau_k}=\lambda (F_{\hat\tau_{k}}-F_{\hat\tau_{k-1}}) +(1-\lambda) (\tilde
F_{\hat\tau_{k}}-\tilde F_{\hat\tau_{k-1}}).
$$
Since $F$ and $\tilde F$ agree on $[0,\hat\tau_{k})$ and $\hat{\gamma}$ is 
predictable,
$\hat{\gamma}_u(F)=\hat{\gamma}_u(\tilde F)$ for all $u \leq \hat\tau_{k}$.
Also, for $u \in (\tau_k,t) \subset (\tau_k, \tau_{k+1})$,
$\gamma^{(N)}_u(S)=\hat \gamma_{\hat \tau_k}(F)$ and
\begin{eqnarray*}
\int_0^t \gamma_u^{(N)}(S) dS_u &=&
\int_0^{\tau_k}\gamma_u^{(N)}(S) dS_u +
\int_{\tau_k}^t \gamma_u^{(N)}(S) dS_u \\
 &=&
\int_{[0,\hat \tau_{k-1}]}\hat \gamma_u(F) dF_u +
\hat \gamma_{\hat \tau_k}(F) [S_t - S_{\tau_k}].
\end{eqnarray*}
Since $F$ is piece-wise constant with jumps
only at
the stopping times $\hat \tau_i$'s,
\begin{eqnarray*}
\int_{[0, \hat \tau_k]} \hat \gamma_u(F) dF_u &=&
\int_{[0,\hat \tau_{k-1}]}\hat \gamma_u(F) dF_u +
\int_{(\hat \tau_{k-1},\hat \tau_k]}\hat \gamma_u(F) dF_u\\
&=&\int_{[0,\hat \tau_{k-1}]}\hat \gamma_u(F) dF_u +
\hat \gamma_{\hat \tau_k}
[F_{\hat \tau_k}-F_{\hat \tau_{k-1}}].
\end{eqnarray*}
We calculate the same integral for $\tilde F$
using the fact that $F=\tilde F$ on $[0,\tau_k)$.
The result is
$$
\int_{[0, \hat \tau_k]} \hat \gamma_u(\tilde F) d\tilde F_u =
\int_{[0,\hat \tau_{k-1}]}\hat \gamma_u(F) dF_u +
\hat \gamma_{\hat \tau_k}
[\tilde F_{\hat \tau_k}-\tilde F_{\hat \tau_{k-1}}].
$$
Therefore,
$$
\int_0^{t} \gamma^{(N)}_u(S) dS_u =
\lambda\int_{[0,{\hat\tau_{k}}]} \hat{\gamma}_{u}(F) dF_u 
+(1-\lambda) \int_{[0,\hat\tau_{k}]} \hat{\gamma}_{u}(\tilde F) d\tilde F_u.
$$
Since $F, \tilde F \in \D^{(N)}$ and
$\hat \P^{(N)}(F), \hat \P^{(N)}(\tilde F) >0$,
 \reff{3.0} imply that
$$
\int_{[0,{\hat\tau_{k}}]}
\hat{\gamma}_{u}(F) dF_u \ge -M,\quad {\mbox{and}} 
\quad \int_{[0,\hat\tau_{k}]} \hat{\gamma}_{u}(\tilde F) d\tilde F_u \ge -M.
$$
Hence, $\gamma^{(N)}$ satisfies
\reff{2.6+} and $\pi^{(N)} \in \cA_N$.
\vspace{10pt}

{\em{Super-replication.}}
We need to show that 
$$ 
g^{(N)}(S_T) + \int_0^T \gamma^{(N)}_u(S) dS_u \ge G^{(N)}(S). 
$$ 
We proceed almost exactly as in the proof of admissibility. Again
we define a modification $\bar F\in \D^{(N)}$ 
by $\bar{F}_{[0,\hat\tau_{n-2})}=F_{[0,\hat\tau_{n-2})}$ 
and $\bar{F}_u=\bar{F}_{\hat\tau_{n-2}}$ for 
$u\geq
\hat\tau_{n-2}$. Set 
$$ 
\hat\lambda:=N |S_T-S_{\tau_{n-1}}|. 
$$ 
Then $\hat\lambda \in [0,1]$ and by the construction of $g^{(N)}$, $$ g^{(N)}(S_T) =
\hat\lambda h(F_T) +(1-\hat\lambda) h(\bar F_T). $$ Hence, 
\begin{eqnarray*} 
g^{(N)}(S_T) &+ & \int_0^T \gamma^{(N)}_u(S) dS_u \\
&=& \hat\lambda\left[
h(F_T) + \int_0^T \hat{\gamma}_{u}(F) dF_u\right] 
+ (1-\hat\lambda) \left[h(\bar F_T) 
+ \int_0^T \hat{\gamma}_{u}(\bar F) d\bar F_u\right] \\ 
&\ge &
\hat\lambda G(F) + (1-\hat\lambda) G(\bar F). 
\end{eqnarray*} 
Since $\|F-\bar F\| \le 1/N$, Assumption \ref{asm2.1} and Lemma \ref{l.1} imply that 
$$
\left|G(S)-G(\bar F)\right| \ge \left|G(S)-G( F)\right| 
+ \left|G(F)-G(\bar F)\right| \le  \frac{5L\|S\|}{N}. 
$$ 
Consequently, 
$$ 
\hat\lambda G(F) +
(1-\hat\lambda) G(\bar F) \ge G^{(N)}(S) 
$$ 
and we conclude that $\pi^{(N)}$ is super-replication $G^{(N)}$. 
\vspace{10pt}

{\em Completion of the proof.}
We have shown that 
$$ 
V_N(G-5 L \|S\|/N) \le \hat V_N(G). 
$$
 Moreover, the linearity 
of the market yields that the super-replication cost is sub-additive.
Hence, 
$$ 
V_N(G) \le V_N(5 L \|S\|/N)+ V_N(G-5L \|S\|/N). 
$$ 
Therefore, 
$$ 
V_N(G) \le V_N(5L \|S\|/N)+  \hat V_N(G). 
$$ 
Finally, by Lemma \ref{l.bound}, 
$$
\limsup_{N \to  \infty} \ V_N(5L \|S\|/N) =0. 
$$ 
We use the above inequalities to complete the proof of the lemma. 
\end{proof}

\subsection{First duality}

Recall the countable set $\D^{(N)}\subset \hat \Omega$ and its probabilistic structure 
were introduced in subsection \ref{ss.countable}. We consider two
classes of measures on this set.

\begin{dfn} \label{d.mart} {\rm{1.  We say that a probability measure $\mathbb{Q}$ on 
the space $(\hat\Omega,\hat{\mathcal F})$ is a}} martingale measure
{\rm{if the canonical process $(\hat{\mathbb S}_t)_{t=0}^T$ is a local martingale with 
respect to $\mathbb Q$.

\noindent 2.   $\mathbb{M}_N$ is the set of all martingale measures that are 
supported on $\D^{(N)}$.

\noindent 3.  For a given $K>0$, $\mathbb{M}^{(K)}_{N}$ is the set of all measures $
\mathbb{Q}\in \M_N$ that satisfy}} \begin{equation}\label{3.18+}
\sum_{k=0}^\infty\left|\mathbb{Q}\left({{\hat{\mathbb S}}}_T=k/N\right) -\mu^{(N)}
\left(\left\{k/N\right\}\right)\right|<\frac{K}{N}. \end{equation} \qed
\end{dfn}

The following follows from known duality results. We will combine it with Lemma 
\ref{lem3.1} and Proposition \ref{lem4.2}, which will be proved in the next
section to complete the proof of the inequality (\ref{2.10}).

\begin{lem}
\label{lem3.2}
Suppose that $G \ge 0$ is bounded from above by $K$ and satisfies the Assumption 
\ref{asm2.1}. Then, for any positive integer
$N$,
$$
\hat{V}_N(G) \leq \sup_{\mathbb Q\in
\mathbb{M}^{(K)}_{N}} \mathbb{E}_{\mathbb Q} \left[G(\hat{\mathbb S})\right].
$$
\end{lem}
\begin{proof}
Fix
$N$ and define the set
$$
\cZ= \cZ^{(N)}:=\{h: A^{(N)}\to \mathbb{R}: |h(x)|\leq N, \ \ \forall{x}\}.
$$
Set
$$
\mathbb V:=\inf_{h\in \cZ} \sup_{\mathbb
Q\in \mathbb{M}_N} \left(\mathbb{E}_{\mathbb Q} (G(\hat{\mathbb S})- 
h(\hat{\mathbb S}_T))+\int h d\mu^{(N)}\right).
$$
Clearly, for any $\epsilon>0$, there
exists $h_\epsilon\in \cZ$ such that
$$
\sup_{\mathbb Q\in \mathbb{M}_N} \mathbb{E}_{\mathbb Q}
\left(G(\hat{\mathbb S})- h_\epsilon(\hat{\mathbb S}_T)\right)
+\int h_\epsilon
d\mu^{(N)} \le \mathbb{V}+\epsilon.
$$
By construction, the support of the measure $\hat \P^{(N)}$ is  $\D^{(N)}$.
Also all elements of $\D^{(N)}$
are piece-wise constant. Therefore, under $\hat \P^{(N)}$
the canonical process $\hat{\mathbb S}$ is trivially a semi-martingale
and we may use
the results of the seminal paper \cite{ds}.
In particular, by Theorem 5.7   in \cite{ds}, for
$$
x\le
\sup_{\mathbb Q\in \mathbb{M}_N} \mathbb{E}_{\mathbb Q}
\left(G(\hat {\mathbb S})- h_\epsilon(\hat{\mathbb S}_T)\right)
+ \epsilon,
$$
there exists an \emph{admissible} portfolio strategy $\hat\gamma$ such that
$$
x+\int_{0}^T
\hat\gamma_u d\hat{\mathbb S}_u\geq G(\hat{\mathbb S})-
h_\epsilon(\hat{\mathbb S}_T), \ \ \hat{\P}^{(N)} \ \ \mbox{a.s.}
$$
Therefore, $(h_\epsilon+x,\hat\gamma)$
satisfies (\ref{3.0})--(\ref{3.1}), consequently
$$
\hat V_N(G) \le x+ \int h_\epsilon d\mu^{(N)} \le
\sup_{\mathbb Q\in \mathbb{M}_N} \mathbb{E}_{\mathbb Q}
\left(G(\hat {\mathbb S})- h_\epsilon(\hat{\mathbb S}_T)\right)
+ \int h_\epsilon d\mu^{(N)} + \epsilon
\le  \mathbb{V}+2\epsilon.
$$
We now let
$\epsilon$ to zero to conclude that
\begin{equation}
\label{3.19}
\hat{ V}_N(G)\leq\ \inf_{h\in Z }\sup_{\mathbb Q\in \mathbb{M}_N} 
\left(\mathbb{E}_{\mathbb Q}(G(\hat{\mathbb S}) -h(\hat{\mathbb S}_T))+\int h
d\mu^{(N)}\right).
\end{equation}

The next step is to interchange
 the order of the above infimum and supremum.
Consider the vector space $\mathbb{R}^{A^{(N)}}$
of all functions $f: A^{(N)} \to \mathbb{R}$ equipped
with the topology of point-wise convergence.
Clearly,
this space is locally convex. Also, since  $A^{(N)}$
is countable, $\cZ$ is a compact subset of $\mathbb{R}^{A^{(N)}}$. The set
$\mathbb{M}_N$ can be
naturally considered as a convex subspace of the vector space
$\mathbb{R^{\D^{(N)}}}$.

Now, define the function
$\mathcal G:\cZ\times \mathbb{M}_N\rightarrow\mathbb{R}$, by
$$
\mathcal G(h,\mathbb Q)=\mathbb{E}_{\mathbb Q}
\left(G(\hat{\mathbb S})-h(\hat {\mathbb S}_T)\right)+\int h d\mu^{(N)}.
$$
Notice that $\mathcal G$ is affine in each of the variables.
From the bounded
convergence theorem, it follows that
$\mathcal G$ is continuous in the first variable.
Next, we apply the min-max theorem,  Theorem 45.8 in \cite{STR} to
$\mathcal G$.  The result is,
$$
\inf_{h\in \cZ }\sup_{\mathbb Q\in \mathbb{M}_N}\mathcal G(h,\mathbb Q)
=\sup_{\mathbb Q\in \mathbb{M}_N}
\inf_ {h\in \cZ}\mathcal G(h,\mathbb Q).
$$
This together with (\ref{3.19}) yields,
\begin{equation}
\label{3.20}
\hat{V}_N(G) \leq \sup_{\mathbb Q\in
\mathbb{M}_N}\inf_{h\in \cZ }
\left(\mathbb{E}_{\mathbb Q}(G(\hat{\mathbb S}) -h(\hat{\mathbb S}_T))
+\int h d\mu^{(N)}\right).
\end{equation}
Finally, for
any measure $\mathbb{Q}\in\mathbb{M}_N$,
define $h^{\mathbb Q}\in \cZ$ by
$$
h^{\mathbb Q}\left(k/N\right)
= N sign \left(\mathbb{Q}\left({{\hat{\mathbb S}}}_T
=k/N\right) -\mu^{(N)}\left(\left\{k/N\right\}\right)\right),
\ \ k=0,1,\ldots.
$$
In view of (\ref{3.20}),
\begin{eqnarray*}
\hat{V}_N(G) &\leq
&\sup_{\mathbb Q\in \mathbb{M}_N}
\left(\mathbb{E}_{\mathbb Q} (G(\hat{\mathbb S}))
+ \int h^{\mathbb Q}_n d\mu^{(N)}
-\mathbb{E}_{\mathbb Q} h^{\mathbb Q}(\hat {\mathbb S}_T)\right)
\\ &=&
\sup_{\mathbb Q\in \mathbb{M}_N}
\left\{\mathbb{E}_{\mathbb Q} (G(\hat{\mathbb B}))
- N \sum_{k=0}^\infty
\left|\mathbb{Q}\left({{\hat{\mathbb S}}}_T=k/N\right)
-\mu^{(N)}\left(\left\{k/N \right\}\right)\right|\right\}
\end{eqnarray*}
Suppose that $\Q \not \in \mathbb{M}^{(K)}_{N}$.  Then,
$$
N \sum_{k=0}^\infty \left|\mathbb{Q}\left({{\hat{\mathbb S}}}_T=k/N\right)
-\mu^{(N)}\left(\left\{k/N \right\}\right)\right| \ge  K.
$$
Since $G$ is bounded by $K$, this implies that
$$
\mathbb{E}_{\mathbb Q} (G(\hat{\mathbb B}))-
N \sum_{k=0}^\infty \left|\mathbb{Q}\left({{\hat{\mathbb S}}}_T
=k/N\right) -\mu^{(N)}\left(\left\{k/N \right\}\right)\right| \le 0.
$$
Since $G \ge 0$,  $\hat V_N(G) \ge 0$ as well.  Hence,
 we may assume that
$\Q \in \mathbb{M}^{(K)}_{N}$. Then,
\begin{eqnarray*}
\hat{V}_N(G)
&\leq & \sup_{\mathbb Q\in \mathbb{M}^{(K)}_N}
\left\{\mathbb{E}_{\mathbb Q}
(G(\hat{\mathbb B}))- N \sum_{k=0}^\infty
\left|\mathbb{Q}\left({{\hat{\mathbb S}}}_T=k/N\right)
-\mu^{(N)}\left(\left\{k/N \right\}\right)\right|\right\}
\\ &\le&
\sup_{\mathbb Q\in \mathbb{M}^{(K)}_{N}}
\mathbb{E}_{\mathbb Q}(G(\hat{\mathbb S})).
\end{eqnarray*}
\end{proof}
\vspace{10pt}

\section{Approximation of Martingale Measures}

In this final section, we prove the asymptotic connection between the approximating martingale measures $\mathbb{M}^{(K)}_{N}$ defined in Definition
\ref{d.mart} and the continuous martingale measures $\mathbb{M}_\mu$
 satisfying the marginal constraint at the final time,
defined in Definition \ref{d.measure}.

The following proposition completes the proof of the inequality \reff{2.10} and consequently the proofs of the main theorems when the claim $G \ge 0$ is
bounded from above. The general case then follows from Lemma \ref{l.bounded}.

\begin{prop}\label{lem4.2} Suppose that $G \ge 0$ is bounded from above by $K$ and satisfies the Assumption \ref{asm2.1}. Assume that $\mu$ satisfies
\reff{2.2}-\reff{2.3}. Then $$ \limsup_{N\rightarrow\infty}\ \sup_{\mathbb Q\in\mathbb{M}^{(K)}_N}\ \mathbb{E}_{\mathbb Q} \left[G (\hat{\mathbb S})
\right] \ \leq \ \sup_{\Q\in\mathbb{M}_\mu} \ \mathbb{E}_{\Q}\left[ G(\mathbb S)
 \right].
$$ \end{prop}

We prove this result not through a compactness argument as one may expect. Instead, we show that any given measure $\Q \in \mathbb{M}^{(K)}_{N}$ has a
lifted version in $\mathbb{M}_\mu$ that is close to $\Q$ in some sense. Indeed, the above proposition is a direct consequence of the below lemma.

Recall the Lipschitz constant $L$ in Assumption \ref{asm2.1}.

\begin{lem}
\label{l.lift}
Under the hypothesis of Proposition \ref{lem4.2}, there exists  a
 function $f_{K}(\epsilon,N)$
satisfying,
 $$
 \lim_{\epsilon\downarrow 0}\lim_{N\rightarrow\infty}f_{K}(\epsilon,N)=0
 $$
 so that
 for any $\hat{\mathbb Q} \in \mathbb{M}^{(K)}_{N}$
and $\epsilon>0$,
$$
\mathbb{E}_{\hat{\mathbb Q}} \left[ G(\hat{\mathbb S})\right] \le f_{K} ( \epsilon,N) +\sup_{\Q\in \mathbb{M}_{\mu}}
\mathbb{E}_{\Q}\left[ G(\mathbb S)\right].
$$
\end{lem}
\begin{proof}

Fix  $\epsilon\in (0,1)$, a positive integer $N$ and
$\hat{\mathbb Q}\in\mathbb{M}^{(K)}_N$.
Recall that $G$ is bounded from above by $K$. \vspace{10pt}

{\em Shift of the initial value.} Denote by $\mathbb{D}^{(N)}_1$ the
 set of all functions $f \in
\mathbb{D}[0,T]$ which satisfy $f(0)=1$ and the conditions 2--4 in definition \ref{d.DN}. Define a map $H:\mathbb{D}^{(N)}\rightarrow \mathbb{D}^{(N)}_1$
by $H(f)=f+1-f(0)$. Consider the measure $\mathbb Q_1= H\circ \hat{\mathbb Q}$, 
clearly $\mathbb Q_1$ is a martingale measure. \vspace{10pt}

{\em Jump times.} Since the probability measure $\hat{\mathbb Q}$ is supported on the set $\mathbb{D}^{(N)}$, the canonical process $\hat{\mathbb S}$ is a
purely jump process under $\mathbb Q_1$, with a finite number of jumps. Introduce the jump times by setting $\tau_0=0$ and for $k>0$, $$
\tau_k=\inf\{t>\tau_{k-1}: \hat{\mathbb S}_t \neq \hat{\mathbb S}_{t{\mbox{\tiny{-}}}}\}\wedge T. $$ Next we introduce the largest random time $$ \hat
N:=\min\{k: \tau_k=T\}. $$ Then,  $\hat N<\infty$ almost surely and consequently, there exists a deterministic positive integer $m$ (depending on
$\epsilon$) such that \begin{equation} \label{4.1} \mathbb{Q}_1(\hat N>m)<\epsilon. \end{equation} By the definition of the set $\mathbb{D}^{(N)}$,
 there is a decreasing sequence of strictly
positive numbers $t_k\downarrow 0$, with $t_1=T$, such that for $i=1,...,m$, $$ \tau_i-\tau_{i-1} \in {\{t_k\}}_{k=1}^\infty \ \cup \ {\{0\}}, \ \
\mathbb{Q}_1-a.s. $$ \vspace{10pt}

{\em Wiener space.} Let $(\Omega^W,\mathcal{F}^W,P^W)$ be a complete probability space together with a standard $m+2$--dimensional Brownian motion
$\left\{W_t=\left(W^{(1)}_t,W^{(2))}_t,...,W^{(m+2)}_t\right)\right\}_{t=0}^\infty$, and the natural filtration $\mathcal{F}^W_t=\sigma{\{W_s|s\leq{t}\}}$.
The next step is to construct a martingale $Z$ on the Brownian probability space $(\Omega^W,\mathcal{F}^W,P^W)$ together with a sequence of stopping times
(with respect to the Brownian filtration) $\sigma_1\leq \sigma_2\leq...\leq\sigma_m$ such that the distribution (under the Wiener measure $P^W$) of the
random vector $(\sigma_1,...,\sigma_m,Z_{\sigma_1},...,Z_{\sigma_m})$ is equal to the distribution of the random vector $(\tau_1,...,\tau_m,\hat{\mathbb
S}_{\tau_1},...,\hat{\mathbb S}_{\tau_m})$ under the measure $\mathbb{Q}_1$. Namely, \begin{equation}\label{4.1+}
\left((\sigma_1,...,\sigma_m,Z_{\sigma_1},...,Z_{\sigma_m}),P^W\right)= \left((\tau_1,...,\tau_m,\hat{\mathbb{S}}_{\tau_1},...,
\hat{\mathbb{S}}_{\tau_m}),\mathbb Q_1\right). \end{equation} The construction is done by induction, at each step $k$ we construct the stopping time
$\sigma_k$ and $Z_{\sigma_k}$ such that the conditional probability is the same as in the case of the canonical process $\hat{\mathbb S}$ under the measure
$\mathbb Q_1$. \vspace{10pt}

{\em Construction of $\sigma$'s and $Z$.} For an integer $n$ and given $x_1,\ldots,x_n$, introduce the notation $$ \vec{x}_n:=(x_1,\ldots,x_n). $$ Also set
$$ \mathbb{T}:= {\{t_k\}}_{k=1}^\infty. $$ For $k=1,...,m$,  define the functions $\Psi_k,\Phi_k:\mathbb{T}^k\times\{-1,1\}^{k-1}\rightarrow [0,1]$ by
\begin{equation} \label{4.2} \Psi_k(\vec{\alpha}_k;\vec{\beta}_{k-1}):=\mathbb{Q}_1 \left(\tau_k-\tau_{k-1}\geq\alpha_k\  \big|\ A \right), \end{equation}
where $$ A:= \left\{ \tau_i-\tau_{i-1}=\alpha_i,\ \hat{\mathbb S}_{\tau_i}-\hat{\mathbb S}_{\tau_{i-1}} =\beta_i/ N, \ \ i\leq k-1\right\}, $$ and
\begin{equation} \label{4.3} \Phi_k(\vec{\alpha}_k;\vec{\beta}_{k-1})= \mathbb{Q}_1 \left(\hat{\mathbb S}_{\tau_k}-\hat{\mathbb S}_{\tau_{k-1}}=1/N \ \big|
\ B\ \right), \end{equation} where $$ B= \left\{ \tau_k<T, \tau_j-\tau_{j-1}=\alpha_j,\ \hat{\mathbb S}_{\tau_i}-\hat{\mathbb S}_{\tau_{i-1}} =\beta_i/N, \
j\leq k, i\leq k-1\right\}. $$ As usual we set $\mathbb{Q}_1(\cdot|\emptyset)\equiv 0$. Next, for $k\leq m$, we define the maps
$\Gamma_k,\Theta_k:\mathbb{T}^k \times\{-1,1\}^{k-1}\rightarrow [-\infty,\infty]$, as the unique solutions of the following equations,
\begin{equation}\label{4.4} P^W\left(W^{(1)}_{\alpha_k}<\Gamma_k( \vec{\alpha}_k;\vec{\beta}_{k-1}) \right)= \Phi_k(\vec{\alpha}_k;\vec{\beta}_{k-1}),
\end{equation} and \begin{equation} \label{4.5} P^W\left(W^{(1)}_{t_l}-W^{(1)}_{t_{l+1}}<\Theta_k( \vec{\alpha}_k;\vec{\beta}_{k-1})\right)=
\frac{\Psi_k(\vec{\alpha}_{k-1},t_l;\vec{\beta}_{k-1})} {\Psi_k(\vec{\alpha}_{k-1},t_{l+1};\vec{\beta}_{k-1})}, \end{equation} where $l\in\mathbb{N}$ is
given by $\alpha_k=t_l\in \mathbb{T}$. From the definitions it follows that $\Psi_k(\vec{\alpha}_{k-1},t_l;\vec{\beta}_{k-1})\leq
\Psi_k(\vec{\alpha}_{k-1},t_{l+1};\vec{\beta}_{k-1})$. Thus if $\Psi_k(\vec{\alpha}_{k-1},t_{l+1};\vec{\beta}_{k-1})=0$ for some $l$, then also
$\Psi_k(\vec{\alpha}_{k-1},t_{l};\vec{\beta}_{k-1})=0$. We set $0/0\equiv 0$.

Set $\sigma_0\equiv0$ and define the random variables $\sigma_1,...,\sigma_m,Y_1,...,Y_m$ by the following recursive relations \begin{eqnarray} \label{4.6}
\sigma_1&=&\sum_{k=1}^\infty t_k \chi_{\{W^{(1)}_{t_k}-W^{(1)}_{t_{k+1}}>\Theta_1(t_k)\}} \ \prod_{j=k+1}^\infty
 \chi_{\{W^{(1)}_{t_j}-W^{(1)}_{t_{j+1}}<\Theta_1(t_j)\}},\\
Y_1&=&2\chi_{\{W^{(2)}_{\{\sigma_1}>\Gamma_1(\sigma_1)\}}-1,\nonumber \end{eqnarray} and for $i>1$ \begin{eqnarray} \label{4.7}
\sigma_i&=&\sigma_{i-1}+\Delta_i\nonumber\\ Y_i&=&\chi_{\{\sigma_i<T\}}\left(2\chi_{\{W^{(i+1)}_{\sigma_i}-W^{(i+1)}_{\sigma_{i-1}}>
\Gamma_i(\vec{\Delta\sigma}_{i},\vec{Y}_{i-1})\}}-1\right),\nonumber \end{eqnarray} where $\Delta_i = t_k$ on the set $A_i \cap B_{i,k}\cap C_{i,k}$ and
zero otherwise.  These sets are given by, \begin{eqnarray*} A_i&:=& {\{|Y_{i-1}|>0\}},\\
B_{i,k}&:=&\{W^{(1)}_{t_k+\sigma_{i-1}}-W^{(1)}_{t_{k+1}+\sigma_{i-1}}> \Theta_i(\vec{\sigma}_{i-1},t_k;\vec{Y}_{i-1})\},\\
C_{i,k}&:=&\bigcap_{j=k+1}^\infty \{W^{(1)}_{t_j+\sigma_{i-1}}-W^{(1)}_{t_{j+1}+\sigma_{i-1}}< \Theta_i(\vec{\Delta\sigma}_{i-1},t_j;\vec{Y}_{i-1})\}.
\end{eqnarray*}

Since $t_k$ is decreasing with $t_1=T$,
 $\sigma_1\leq \sigma_2\leq...\leq\sigma_m$ and
 they are stopping times with respect to the
Brownian filtration. Let $k\leq m$ and $(\vec{\alpha}_k;\vec{\beta}_{k-1})\in \mathbb{T}^k\times\{-1,1\}^{k-1}$. There exists $m\in\mathbb{N}$ such that
$\alpha_k=t_m\in\mathbb{T}$. From (\ref{4.6})--(\ref{4.7}), the strong Markov property and the independency of the Brownian motion increments it follows
that \begin{eqnarray}\label{4.8} &&P^W(\sigma_k-\sigma_{k-1}\geq\alpha_k \big| (\vec{\Delta
\sigma}_{k-1};\vec{Y}_{k-1})=(\vec{\alpha}_{k-1};\vec{\beta}_{k-1}))
 \\
&&\hspace{40pt}= P^W\bigg(\bigcap_{j=m}^\infty \bigg(W^{(1)}_{t_j+\sigma_{k-1}}-W^{(1)}_{t_{j+1}+\sigma_{k-1}}
<\Theta_k(\vec{\alpha}_{k-1},t_j;\vec{\beta}_{k-1})\bigg)\bigg) \nonumber\\ &&\hspace{40pt}=\prod_{j=m}^\infty
P^W\bigg(W^{(1)}_{t_j+\sigma_{k-1}}-W^{(1)}_{t_{j+1}+\sigma_{k-1}}< \Theta_k(\vec{\alpha}_{k-1},t_j;\vec{Y}_{k-1})\bigg)\nonumber\\
&&\hspace{40pt}=\Psi_k(\vec{\alpha}_k,\vec{\beta}_{k-1}), \nonumber \end{eqnarray} where the last equality follows from (\ref{4.5}) and the fact that $$
\lim_{l\rightarrow\infty} \Psi_k (\alpha_1,...,\alpha_{k-1},t_l,\beta_1,...,\beta_{k-1})=1. $$ Similarly, from (\ref{4.4}) and (\ref{4.7}), we have
\begin{eqnarray} \label{4.9} &&P^W\left(Y_k=1\big|\sigma_k<T,\vec{\Delta \sigma}_k=\vec{\alpha}_k,\vec{Y}_{k-1}=\vec{\beta}_{k-1}\right) \\
&&\hspace{60pt}=P^W\left(W^{(k+1)}_{\sum_{i=1}^k \alpha_i}-W^{(k+1)}_{\sum_{i=1}^{k-1} \alpha_i}<
\Gamma_k(\vec{\alpha}_k;\vec{\beta}_{k-1})\right)\nonumber\\ &&\hspace{60pt}=\Phi_k(\vec{\alpha}_k;\vec{\beta}_{k-1}).\nonumber \end{eqnarray} Using
(\ref{4.2})--(\ref{4.3}) and (\ref{4.8})--(\ref{4.9}), we conclude that $$ \left((\vec{\sigma}_m;\frac{1}{N}\vec{Y}_m), P^W\right)=\left((\vec{\tau}_m;
\vec{\Delta \hat{\mathbb S}}_m),\mathbb{Q}_1\right) $$ where ${\Delta \hat{\mathbb S}}_k=\hat{\mathbb{S}}_{\tau_k} -\hat{\mathbb{S}}_{\tau_{k-1}}$, $k\leq
m$. \vspace{10pt}

{\em Continuous martingale.} Set \begin{equation}\label{4.11} Z_t=1+\frac{1}{N}E^W(\sum_{i=1}^m Y_i|\mathcal{F}^W_t), \ \ t\in [0,T]. \end{equation} Since
all Brownian martingales are continuous, so is $Z$. Moreover, Brownian motion increments are independent and therefore, \begin{equation}\label{4.12}
Z_{\sigma_k}=1+\frac{1}{N}\sum_{i=1}^k Y_i, \ \ P^W\mbox{a.s.}, \ \ k\leq m. \end{equation} By the construction of $Y$ and $\sigma$'s, we conclude that
(\ref{4.1+}) holds with the process $Z$. \vspace{10pt}

{\em Measure in $\M_{\mu}$.} The next step in the proof is to modify the martingale $Z$ in such way that the distribution of the modified martingale is an
element of $\mathbb{M}_{\mu}$. For any two probability measures $\nu_1,\nu_2$ on $\mathbb{R}$, Prokhorov's metric is defined by
$$d(\nu_1,\nu_2)=\inf\{\delta>0: \nu_1(A)\leq \nu_2(A^\delta)+\delta \ \mbox{and} \ \nu_2(A)\leq \nu_1(A^\delta)+\delta, \ \
\forall{A}\in\mathcal{B}(\mathbb{R})\},$$ where $\mathcal{B}(\mathbb{R})$ is the set of all Borel sets $A\subset\mathbb{R}$ and $A^\delta:=\bigcup_{x\in
A}(x-\delta,x+\delta)$ is the $\delta$--neighborhood of $A$. It is well known that convergence in the Prokhorov metric is equivalent to weak convergence,
(for more details on Prokhorov's metric see \cite{S}, Chapter 3, Section 7).

Let $\nu_1$ and $\nu_2$, be the distributions of $\hat{\mathbb S}_{\tau_m}$ and $\hat{\mathbb{S}}_T$ respectively, under the measure $\mathbb{Q}_1$. Let
$\nu_3$ be the be the distributions of $\hat{\mathbb S}_T$ under the measure $\hat{\mathbb Q}$. In view of (\ref{4.1}), $d(\nu_1,\nu_2)<\epsilon$. From the
definition of the measure $\mathbb Q_1$ it follows that $d(\nu_2,\nu_3)<\frac{2}{N}$. Moreover, (\ref{3.18+}) implies
 that $d(\nu_3,\mu^{(N)})<\frac{K}{N}$
and $\mu^{(N)}$ converges to $\mu$ weakly. Hence, the preceding inequalities, together with this convergence yield that for all sufficiently large $N$,
$d(\nu_1,\mu)<2\epsilon$. Finally, we observe that in view of  (\ref{4.1+}), $(Z_T,P^W)=\nu_1$.

We now use Theorem 4 on page 358 in \cite{S} and Theorem 1 in \cite{Sk} to construct a measurable function $\psi:\mathbb{R}^2\rightarrow\mathbb{R}$ such
that  the random variable $\Lambda:=\psi(Z_T,W^{(m+2)}_T)$ satisfies \begin{equation}\label{4.16} (\Lambda,P^W)=\mu \ \ \mbox{and} \ \
P^W(|\Lambda-Z_T|>2\epsilon|)<2\epsilon. \end{equation} We define a martingale by, $$ \Gamma_t=E^W(\Lambda|\mathcal{F}^W_t), \ \ t\in [0,T]. $$ In view of
(\ref{4.16}),  the distribution of the martingale $\Gamma$ is an element in $\mathbb{M}_{\mu}$. Hence, \begin{equation}\label{4.16++} \sup_{\mathbb
Q\in\mathbb{M}_\mu} \mathbb{E}_{\Q} \left[G(\mathbb S)\right] \geq E^W(G(\Gamma)). \end{equation}

We continue with the estimate that connects the distribution of $\Gamma$ to $\Q \in \mathbb{M}^{(K)}_{N}$. Observe that $E^W\Lambda=E^W Z_T=1$. This
together with (\ref{4.16}), positivity of $Z+\frac{1}{N}$ and $\Lambda$, and the Holder inequality yields \begin{eqnarray} \label{4.16+++}
E^W|\Lambda-Z_T|&=& 2 E^W(\Lambda-Z_T)^{+}-E^W(\Lambda-Z_T)\\ \nonumber &=& 2 E^W(\Lambda-Z_T)^{+} \\ \nonumber &\leq& 4\epsilon+\frac{2}{N}+ 2E^W(\Lambda
\chi_{\{|\Lambda-Z_T|>2\epsilon\}})\\ &\leq&
 4\epsilon+\frac{2}{N}+2(\int x^p
 d\mu(x))^{1/p}(2\epsilon)^{1/q},\nonumber
\end{eqnarray} where $p>1$ is as \reff{2.3} and $q=p/({p-1})$. From (\ref{4.16+++}) and the Doob inequality for the martingale $\Gamma_t-Z_t$, $t\in [0,1]$
we obtain \begin{equation}\label{4.16++++} E^W(\chi_{\{\|\Gamma-Z\|>\epsilon^{1/2q}\}})\leq \frac{E^W|\Lambda-Z_T|}{\epsilon^{1/2q}}\leq
\frac{4\epsilon+\frac{2}{N}+2\left(\int x^p d\mu(x)\right)^{1/p}(2\epsilon)^{1/q}}{\epsilon^{1/2q}}. \end{equation}

We now introduce a stochastic process ${(\hat{Z}_t)}_{t=0}^T$, on the Brownian probability space, by, $\hat{Z}_t=Z_{\sigma_k}$ for $t\in
[\sigma_k,\sigma_{k+1})$, $k<m$ and for $t\in [\sigma_m,T]$, we set $\hat{Z}_t=Z_{\sigma_m}$. On the space $(\hat\Omega,\mathbb Q_1)$ let $\tilde{\mathbb
S}_t=\hat{\mathbb S}_{t\wedge \tau_m}$, $t\in [0,T]$. Recall that $G$ is bounded by $K$.  We now use the Assumption (\ref{asm2.1}) together with
(\ref{4.1}) and (\ref{4.12}) to arrive at \begin{eqnarray} \label{4.13} &&\mathbb{E}_{\mathbb Q_1}( G(\hat{\mathbb S}))- \mathbb{E}_{\mathbb Q_1}
(G(\tilde{\mathbb S}))\leq K \epsilon \\ &&|E^W(G(Z))-E^W (G(\hat Z))| \leq L E^W \| Z-\hat Z\|\leq \frac{L}{N}.\nonumber \end{eqnarray}

Recall that by (\ref{4.1+}), $(\hat Z,P^W)=(\tilde{\mathbb {S}},\mathbb{Q}_1)$. Thus, $E^W (G(\hat Z))= \mathbb{E}_{\mathbb Q_1}(G(\tilde{\mathbb S})).$
This together with Assumption \ref{asm2.1} and (\ref{4.13}) yields \begin{equation} \label{4.14} \mathbb{E}_{\hat{\mathbb Q}}(G(\hat{\mathbb S}))\leq
\frac{L}{N}+\mathbb{E}_{\mathbb Q_1}(G(\hat{\mathbb S})) \leq \frac{2L }{N}+K  \epsilon+ E^W(G(Z)). \end{equation}

From Assumption \ref{asm2.1}, (\ref{4.16++})--(\ref{4.16++++}) and (\ref{4.14}) we obtain \begin{eqnarray*} \sup_{\Q\in\mathbb{M}_\mu}
\mathbb{E}_{\Q}\left[ G(\mathbb S)\right]&\geq& E^W(G(\Gamma))\\ &\geq &E^W(G(Z))-L \epsilon^{1/2q}- K
E^W(\chi_{\{\|\Gamma-Z\|>\epsilon^{1/2q}\}})\\ &\geq& \mathbb{E}_{\hat{\mathbb Q}}(G(\hat{\mathbb S}))- f_{K} (\epsilon,N), \end{eqnarray*} where $$
f_{K}(\epsilon,N) =\frac{2L}{N}+ K \epsilon+ L\epsilon^{1/2q}+K \frac{4\epsilon+\frac{2}{N}+\left(\int x^p
d\mu(x)\right)^{1/p}(2\epsilon)^{1/q}}{\epsilon^{1/2q}}. $$ \end{proof} \vspace{10pt}

\section{Possible extensions.}
\label{s.extension}
In this paper, we prove a Kantorovich type duality
for a super-replication problem
in financial market
with no prior probability structure.
The dual is a martingale optimal
problem.

The main theorem
holds for nonlinear
path-dependent options
satisfying Assumption \ref{asm2.1}.
Although this condition is satisfied
by most of the examples,
it is an interesting question
to characterize the class of functions
for which the duality holds.
A possible procedure
for extending the proof is the following.
 Assumption \ref{asm2.1}
 is used in the proofs of
 Lemmas \ref{l.bounded}, \ref{l.1} and \ref{lem3.1}.
 In Lemma \ref{l.bounded},
 only the linear growth implied by the assumption
 is used and one may replace this
 assumption by an appropriate
 growth condition on the function $G$.
 In particular, if $G$ is bounded no assumption
 would be required.

Since the inequality
(\ref{2.10++}) is holds for any measurable function $G$,
we need to extend the proof
of  the inequality (\ref{2.10}).
We may achieve this
by modifying the right hand side of formula
\reff{3.1}  in Definition \ref{dfn3.1}
and use a sequence of functions
$G_n(\hat{\mathbb S})$ satisfying the
Assumption \ref{asm2.1}
and   $G_n\downarrow G$
 as $n$ approaches to $\infty$.
 Under this structure,  we
 skip Lemma \ref{l.1}, and  prove Lemma
  \ref{lem3.1} directly.
  The final step would
  be a modification of Proposition \ref{lem4.2}
  to the following claim
  $$
\limsup_{N\rightarrow\infty}\
\sup_{\mathbb Q\in\mathbb{M}^{(K)}_N}\ \mathbb{E}_{\mathbb Q}
\left[G_n (\hat{\mathbb S}) \right] \ \leq \
\sup_{\Q \in\mathbb{M}_\mu} \ \mathbb{E}_{\Q}\left[ G(\mathbb S)
 \right].
$$

 This extension technique also applies to Barrier options.
 In this case, we use the approximating sequence
 as the payoffs $G_n$ of
 Barrier options with a larger (than the
 original payoff $G$) corridor.
 The main concern here is 
 to discretize the process in a way
 adapted to the barriers.
 
 Two other  important extensions
are to the case of many stocks
and the inclusion of the possibility
of jumps into the stock price process.
 We believe that for the multi-dimensional case,
 a  discretization based
 proof would be possible.
The main difficulty here is 
to appropriately define the crossing times
and use them to obtain
a piece-wise constant 
approximation of a generic 
stock price process.


\begin{thebibliography}{}

\bibitem{AP} L.~Ambrosio and A.~Pratelli, {\em Existence and stability results in the $L^1$ theory of optimal
    transportation} In {\em Optimal
    transportation and applications} (Martina Franca, 2001),
volume 1813 of Lecture Notes in Math., pages 123Ð160.
    Springer, Berlin, (2003)

\bibitem{ABPST} B.~Acciaio, M.~Beiglb\"ock, F.~Penkner,
    W.~Schachermayer and J.~Temme, {\em A Trajectorial
    Interpretation of Doob's Martingale Inequalities,}
    to appear in Ann. Appl. Prob.

    \bibitem{ABS}
B.~Acciaio, M.~Beiglb\"ock and W.~Schachermayer,
{\em Model--free versions of the fundamental theorem of
asset pricing and the super--replication theorem,}
preprint.


\bibitem{B} P.~Billingsley, {\em Convergence of Probability Measures,} 
Wiley, New York, (1968).


\bibitem{BHLP} M.~Beiglb\"ock, P.~Henry-Labord\`ere and
    F.~Penkner, {\em Model--independent bounds for option
    prices: a mass transport approach}, Finance and Stochastics,
    {\bf 17}, 477--501, (2013).

\bibitem{BHR} H.~Brown, D.~Hobson and L.C.G.~Rogers, {\em
    Robust hedging of barrier options,} Math.Finance, {\bf
    11}, 285--314, (2001).

    \bibitem{BN} B.~Bouchard and N.~Nutz.
    {\em Arbitrage and Duality in Non-dominated Discrete-Time Models},
    preprint.

\bibitem{CL} P.~Carr and R.~Lee, {\em Hedging Variance
    Options on Continuous Semimartingales,} Finance and
    Stochastics, {\bf 14}, 179--207, (2010).
    
 \bibitem{CHO}
 A. M. G.~Cox, D.~Hobson, and J.~Obloj,
 {\em Pathwise inequalities for local time: applications to 
 Skorokhod embeddings
and optimal stopping}, 
 Ann. Appl. Probab., {\bf 18/5}, 1870--1896, (2008).

\bibitem{CO} A.M.G.~Cox and J.~Obloj, {\em Robust pricing
    and hedging of double no-touch options}, Finance and
    Stochastics, {\bf 15}, 573--605, (2011).
   
  \bibitem{CO1} A.M.G.~Cox and J.~Obloj, 
  {\em Robust hedging of double touch barrier options}. 
  SIAM J.~Financial Math., {\bf 2}, 141--182, (2011).

\bibitem{CW} A.M.G.~Cox and Wang, {\em Root's Barrier:
    Construction, Optimality and Applications to Variance
    Options}, Annals of Applied Probability,
    forthcoming, (2012).

\bibitem{DM} L.~Denis and C.~Martini, {\em A Theoretical
    Framework for the Pricing of Contingent Claims in the
    Presence of Model Uncertainty,} Ann. Appl.
    Probab. {\bf 16}, 827--852, (2006).

\bibitem{DH}
M.H.A.~Davis and D.~Hobson. 
{\em The range of traded option prices},
Math. Finance, {\bf 17/1}, 1--14, (2007).

\bibitem{DOR} M.H.A.~Davis, J.~Obloj and V.~Raval {\em
    Arbitrage bounds for prices of weighted variance swaps},
Mathematical Finance, forthcoming, (2012).

\bibitem{ds} F.~Delbaen and W.~Schachermayer, {\em A general version of the fundamental theorem of asset pricing},
    Math. Annalen, {\bf 300}, 463 -- 520,
    (1994).


\bibitem{eq} N.~ El Karoui  and M-C.~Quenez, {\em Dynamic
    programming and pricing of contingent claims in an
    incomplete market}, SIAM J.~Control and Opt.,
    {\bf 33/1}, 29 -- 66, (1995).

\bibitem{GLT} A.~Galichon, P.~Henry-Labord\`ere and
    N.~Touzi, {\em A stochastic control approach to
    no-arbitrage bounds given marginals, with an
    application to Lookback options,} preprint.

\bibitem{H} D.~Hobson, {\em Robust hedging of the lookback
    option,} Finance and Stochastics, {\bf 2}, 329--347,
    (1998).

\bibitem{H1} D.~Hobson, {\em The Skorokhod Embedding Problem and Model-Independent Bounds for Option Prices,}
    Paris--Princeton Lectures on Mathematical
    Finance, Springer, (2010).
    
\bibitem{H2} D.~Hobson and M.~Klimmek, 
{\em Model independent hedging strategies for variance swaps}, 
Finance Stochastics,
{\bf 16}, 611--649, (2012).

\bibitem{H3} D.~Hobson, P.~Laurence and T.H.~ Wang,
{\em  Static-arbitrage optimal sup-replicating strategies for basket
options}, Insur. Math. Econ. {\bf 37}, 553--572, (2005).

\bibitem{H4} D.~Hobson, P.~Laurence and T.H.~ Wang,
{\em Static-arbitrage upper bounds for the prices of basket options},
Quantitative Finance 5, 329--342, (2005).
    
\bibitem{H5} D.~Hobson and A.~Neuberger, 
{\em Robust bounds for forward start options},
Math.~Finance,  {\bf 22}, 31--56, (2012).

\bibitem{H6} D.~Hobson and J.L.~Pedersen, 
{\em The minimum maximum of a continuous martingale with given initial
and terminal laws},  Ann. Probab.,
 {\bf 30}, 978--999, (2002).
 
\bibitem{K} L.V.~Kantorovich, {\em On the transfer of
    masses}, Dokl.~Akad.~Nauk. SSSR (in Russian), {\bf 37},
227--229, (1942).

\bibitem{K1} L.V.~Kantorovich, 
{\em On a problem of Monge}, Uspekhi Mat. Nauk. {\bf 3}, 225--226, (1948).

\bibitem{Kel} H.G. Kellerer {\em Duality theorems for
    marginal problems}, Z. Wahrsch. Verw. Gebiete, {\bf67},
399–432, (1984).

\bibitem{LT} P.~Henry-Labord\`ere and N.~Touzi, {\em Maximum Maximum of Martingales given Marginals}, preprint.

\bibitem{M} G.~Monge, {\em Memoire sur la Theorie des
    D\'eblais et des Remblais}, Histoire de L'Acad. des
    Sciences de Paris, (1781).
    
  \bibitem{O}
  J.~Obloj, {\em The Skorokhod embedding problem and its offspring},
  Probab. Surv., {\bf 1}, 321--390, (2004).

\bibitem{S} A.N.~Shiryaev, {\em Probability,}
    Springer-Verlag, New York, (1984).

\bibitem{Sk} A.V.~Skorokhod, {\em On a representation of
    random variables,} Theory Probab. Appl, {\bf 21},
    628--632, (1976).

\bibitem{STR} H. Strasser, {\em Mathematical theory of
    statistics,} de Gruyter Studies in Mathematics 7.,
    Berlin, 1985.

\bibitem{STZ} H.M.~Soner, N.~Touzi,  and J.~Zhang. {\em Dual Formulation of Second Order Target Problems}, Annals of Applied Probability,
23/1, 308--347, (2013).

\bibitem{STZ1}
 H.M.~Soner, N.~Touzi,  and J.~Zhang.
 {\em Wellposedness of Second Order Backward SDEs},
 Probability Theory and Related Fields, {\bf 153}, 149--190,
(2012).


\bibitem{V} C.~Villani. {\em Topics in optimal
    transportation}, volume 58 of Graduate Studies in
    Mathematics. AMS, Providence, RI, (2003).

\bibitem{V9} C.~Villani. {\em Optimal Transport. Old and
    New}, volume 338 of Grundlehren der mathematischen
    Wissenschaften. Springer, (2009).


\end{thebibliography}
\end{document}